\documentclass[12pt]{amsart}
\usepackage{amssymb}
\usepackage{amsmath}
\usepackage{mathabx}
\usepackage{graphicx}
\usepackage{subfig}
\usepackage[latin1]{inputenc}
\usepackage{enumerate}
\usepackage{color}
\usepackage{url}
\def\toitself{\circlearrowleft}
\def\ti{\tilde}

\def\N{{\mathbb N}}\def\Z{{\mathbb Z}}\def\Q{{\mathbb Q}}\def\T{{\mathbb T}}\def\R{{\mathbb R}}\def\C{{\mathbb C}}
\def\cA{\mathcal{A}}
\def\cE{\mathcal{E}}

\def\cU{\mathcal{U}}

\def\cP{\mathcal{P}}

\def\bS{\mathbb{S}}

\def\be{\begin{equation}}
\def\ee{\end{equation}}

\def\ba{{\begin{align}}}
\def\ea{{\end{align}}}

\def\bm{\begin{pmatrix}}
\def\em{\end{pmatrix}}

\def\a{\alpha}
\def\ph{\varphi}\def\w{\omega}\def\e{\eta}

\def\d{\delta}\def\ep{\varepsilon}
\def\l{\lambda}
\def\th{\theta}

\def\dist{\mathrm{dist}}
\def\supp{\mathrm{supp}}
\def\spec{\mathrm{spec}}
\def\interior{\mathrm{interior}}
\def\tr{\mathrm{tr}}
\def\pa{\partial}

\def\<{\langle}
\def\>{\rangle}

\theoremstyle{plain}
\newtheorem{Thm}{Theorem}[section]

\newtheorem{theo}{Theorem}[section]

\newtheorem{prop}[Thm]{Proposition}

\newtheorem{lem}[Thm]{Lemma}

\newtheorem{cor}[Thm]{Corollary}

\newtheorem{rem}{Remark}

\newtheorem{Main}{Theorem}

\renewcommand{\ti}{\widetilde}

\begin{document}

\title[Coexistence of ac and pp spectrum]{Coexistence of ac and pp spectrum for kicked quasi-periodic potentials}

\author{K. Bjerkl\"ov and R. Krikorian}

\address{
Department of Mathematics, KTH Royal Institute of Technology, Lindstedtsvägen 25, SE-100 44, Stockholm, Sweden
}

\email{ bjerklov@kth.se}

\address{
Department of Mathematics, CNRS UMR 8088,
Universit\'{e} de Cergy-Pontoise,  2, av. Adolphe Chauvin F-95302 Cergy-Pontoise, France } 
\email{raphael.krikorian@u-cergy.fr}


\begin{abstract} We introduce a class of  real  analytic ``peaky''  potentials for which the corresponding quasi-periodic 1D-Schr\"odinger operators exhibit, for quasiperiodic frequencies in a set of positive Lebesgue measure,  both absolutely continuous  and pure point  spectrum. 
\end{abstract}

\thanks{RK is supported by a Chaire d'Excellence LABEX MME-DII and the project ANR BEKAM : ANR-15-CE40-0001.}

\maketitle

\section{Introduction and Results}
\subsection{Quasiperiodic Schr\"odinger operators}
For  $x\in\T:=\R/\Z$, $\a\in (\R\setminus\Q)/\Z$
 and $V:\R/\Z\to\R$  a $C^k$ function, $k\in\N\cup\{\infty,\omega\}$, we consider the following 1D {\it quasiperiodic Schr\"odinger  operator}  
\begin{align*}H_{V,\a,x}&:l^2(\Z)\toitself\\
(u_{n})_{n\in\Z}&\mapsto (u_{n+1}+u_{n-1}+V(x+n\a)u_{n})_{n\in\Z}.
\end{align*}
We shall call  $V$ the {\it  potential},  $\a$ the {\it frequency} and $x$ the {\it phase}. 
The {\it spectrum } $\spec(H_{V,\a,x})$ of the operator $H_{V,\a,x}$ is a nonempty, compact subset of $\R$ and  
 when $\a$ is irrational  it does not  depend on $x$; we shall often denote it by $\Sigma_{V,\a}$. To investigate the spectral properties of the operator $H_{V,\a,x}$ it is often useful to introduce its {\it  spectral measure} $\mu_{V,\a,x}$; this is  the probability measure on $\R$ defined as follows: if $\delta_{i}\in l^2(\Z)$, $i=0,1$ are defined by $\delta_{i}(n)=\delta_{i,n}$, one sets  $\mu_{V,\a,x}=(1/2)(\mu_{V,\a,x,\delta_{0}}+\mu_{V,\a,x,\delta_{1}})$ where the probability measures  $\mu_{V,\a,x,\delta_{i}}$ $i=0,1$ are uniquely defined by 
 $$\forall\ z\in\C,\ \Im z>0,\ \<(H_{V,\a,x}-z)^{-1}\delta_{i},\delta_{i}\>=\int_{\R}\frac{d\mu_{V,\a,x,\delta_{i}}(t)}{t-z}.
 $$
An important question in the theory of quasi-periodic Schr\"odinger operators is to investigate the topological nature of the spectrum $\Sigma_{V,\a}$ (is it a Cantor set?)  and the spectral type of the spectral measure $\mu_{V,\a,x}$ (does it have absolutely continuous, singular continuous or atomic components?). In this paper we shall be interested on the possible coexistence of absolutely continuous (pure ac) and pure point (pp)  components of the spectrum in the following sense: 
 there exist disjoint nonempty  open  intervals $I_{ac}$ and $I_{pp}$ such that $\mu_{V,\a,x}{| I_{ac}}$ is absolutely continuous  (w.r.t. to Lebesgue measure) and $\mu_{V,\a,x}{| I_{pp}}$ is purely punctual. When this is the case we shall say for short that the spectrum of $H_{V,\a,x}$ has {\it disjoint ac and pp components}.

In the preceding questions, the regularity and the size of the potential $V$ on the one hand  and the arithmetic properties of the frequency $\a$ on the other hand  play an important role. Let us mention two general results. Assume that $V$ is of the form $V=\lambda v$ where $v:\T\to \R$ is a fixed  real analytic function  and $\lambda\in\R_{+}$:

\smallskip\noindent {\it Small real analytic potentials, Eliasson's Theorem \cite{El-cmp}}: When $\a$ is {\it diophantine}:
\begin{equation}
 \forall\ (k,l)\in\Z\times\N^*,\ |\a-\frac{k}{l}|\geq \frac{\eta}{l^{\sigma}}\tag{$DC_{1}(\eta,\sigma)$}\label{DC}
\end{equation}
and when $\lambda$ is  small enough (the smallness depending on the preceding diophantine condition), Eliasson has proved that the spectrum of $H_{\l v,\a,x}$ is absolutely continuous. 
His results extends to the case of small {\it smooth} potentials $V=\l v$ and to the multifrequency case ($v:\T^d\to\R$). In the one-frequency case ($d=1$) and when $v$ is real analytic, Bourgain and Jitomirskaya \cite{BoJito} proved that the required smallness  condition on $\lambda$ could be chosen {\it independently} of the diophantine condition.

\smallskip\noindent{\it Large real analytic potentials, Bourgain-Goldstein Theorem \cite{BG}} (see also \cite{El-acta}):   Bourgain and Goldstein proved in \cite{BG} that there exists $\lambda_{0}(v)>0$ and a set of full Lebesgue measure (depending on $v$) such that for any frequency $\a$ in this set and any $\lambda>\lambda_{0}(v)$ the spectrum of $H_{\l v,\a,0}$ is pure point. Moreover, $H_{\lambda v,\a,0}$ satisfies {\it Anderson Localization}: all the eigenvalues of $H_{\lambda v,\a,0}$ decay exponentially fast.  The analytic regularity of $v$ is here essential. If $V=\lambda v$ is only smooth and $\a$ is diophantine one can only establish the existence of {\it some} point spectrum (cf. \cite{Si},   \cite{FSW}, \cite{Su}, \cite{Bj-etds}). 

A fundamental example in the preceding setting  is the case $v=2\cos(2\pi\cdot)$ (one then speak of $H_{\lambda v,\a,x}$ as the {\it Almost Mathieu Operator} (AMO)) where the transition between ``small'' and ``large'' potentials occurs at $\lambda=1$ (\cite{J}). Notice that in this case Anderson Localization of $H_{\lambda v,\a,0}$ for $\lambda>1$ holds for any diophantine frequency $\a$. 

A natural question is to provide examples of analytic potentials and frequencies for which  the spectrum of the corresponding qp Schr\"odinger operator has both absolutely continuous (ac) and purely punctual (pp) spectral components. Natural candidates for such  potentials are perturbations of  the AMO at the critical coupling $\l=1$. Indeed, Avila \cite{A-Acta} constructed examples of potentials which are real analytic  perturbations of $2\cos(2\pi\cdot)$ and for which the spectrum of the corresponding Schr\"odinger operator has both ac and pp components; in fact these ac and pp components can be produced in many alternating intervals on the real axis. Previous examples of  quasi-periodic potentials  with two frequencies where constructed by Bourgain \cite{Bo}. For other types of coexistence results, we refer to \cite{FK} (co-existence of ac and singular spectrum),  \cite{Bj-gafa} (coexistence of regions of the spectrum with  positive Lyapunov exponents and zero Lyapunov exponents) and  \cite{Zh} (which elaborates on   \cite{Bj-gafa} to give examples of coexistence of ac and pp spectrum and coexistence of ac and sc spectrum).

\subsection{Main results}
The aim of this paper is to present  another type of potentials for which coexistence of ac and pp spectrum holds. These potentials are not particularly big (they cannot be too small by Eliasson's Theorem \cite{El-cmp}) but have a ``peaky'' shape. The sets of admissible frequencies for which our Theorems hold have positive Lebesgue measures and are  located close to rational numbers.

More precisely, we define the set $\cP^\infty$ of smooth ``peaky'' potentials $V\in C^\infty(\T,\R)$ by: $V\in\cP^\infty$ if and only if
\begin{enumerate}
\item $V\geq 0$
\item the support $\supp(V)$ of $V$ is a proper subset of $\T$
\item $V$ has a unique maximum at some point $x_{*}$
\item and for any $x\in \interior(\supp(V))\setminus\{x_{*}\}$ one has $V'(x)\ne 0$. 
\end{enumerate}
For $V\in\cP^\infty$ we denote by $L(V)$ the length of the support of $V$ and $K(V)=\max_{\T} V$.
For $\ep>0$ and $k\in\N$ we define $\cP^\omega(V;k,\ep)$ as the set of {\it real analytic} potentials $\ti V:\T\to\R$ such that 
$$ \|V-\ti V\|_{C^k}:= \max_{0\leq l\leq k}\max_{x\in\T}(|\pa^l(V(x)-\ti V(x))|)<\ep.
$$

We now describe the set which will essentially be our set of admissible frequencies. We  define for $\eta>0$, $p/q\in\Q$  the set $D_{p/q}(\eta)$
$$D_{p/q}(\eta)=]\frac{p}{q}-\eta,\frac{p}{q}+\eta[\cap DC_{1}(\eta^2,3).$$
where $DC_{1}(\eta^2,3)$ is  the set of $\a\in\T$ such that 
$$\forall\ (k,l)\in\Z\times\N^*,\ |\a-\frac{k}{l}|\geq \frac{\eta^2}{l^{3}}.$$
The set $D_{p/q}(\eta)$ is a set of positive Lebesgue  measure for $\eta$ small enough, {\it cf.} Lemma \ref{lem:posmeas}. 

Our main Theorem is the following.

\begin{Main}\label{MainA}There exists $s_{0}\in\N^*$ such that the following holds. Let $V\in\cP^\infty$ and $q\in \N^*$ be such that $K(V)>10$ and $L(V)<1/q$. Then, there exists $\ep>0$ such that for any $\ti V\in\cP^\omega(V;s_{0},\ep)$ there exists a  set  of frequencies $\cA_{q}(\ti V)$ of full Lebesgue measure in $\bigcup_{p=0}^{q-1}D_{p/q}(\eta)$ such that for any $\a\in\cA_{q}(\ti V)$ the spectrum of $H_{\ti V,\a,0}$ has  disjoint  a.c. and p.p. components. 
\end{Main}

\begin{rem}The preceding Theorem combined with the acritality result of   Avila \cite{A-Acta} shows  that there is a prevalent set $\ti\cU\subset \cP^\omega(V;s_{0},\ep)$ such that for any   $\ti V\in\ti\cU$ and any $\a\in \cA_{q}(\ti V)$, $H_{V,\a,0}$ has no s.c. spectrum.
\end{rem}
\begin{rem} Let $\cP^{rat}(V;s_{0},\ep)$ be the subset of all functions of the form $P/Q$ in $\cP^\omega(V;s_{0},\ep)$ where $P$ and $Q$ are  trigonometric polynomials. A simple Fubini type argument shows that there exists a  set $\cA_{q}^{rat}$  of full Lebesgue measure in $\bigcup_{p=0}^{q-1}D_{p/q}(\eta)$ such that for all $(\a,\ti V)\in \cA_{q}^{rat}\times \cP^{rat}(V;s_{0},\ep)$, the spectrum of $H_{\ti V,\a,0}$ has  disjoint  a.c. and p.p. components.
\end{rem}

We can also prove  a similar result for a specific class of  examples.

Let $\lambda$ and $K$ be two positive constants and let $V_{K,\lambda}:\R/\Z\to\R$ be the real analytic potential defined by 
$$
V_{K,\lambda}(x)=\frac{K}{1+4\lambda\sin^{2}(\pi x)}.
$$

\begin{Main}\label{MainC}Assume that $K$ and $\l$ are  large enough. Then, there is a set  $\cA_{2}\subset\T$ of positive Lebesgue measure such that  for any $\a\in \cA_{2}$  the operator $H_{V_{K,\l},\a,0}$ has both a.c. and p.p. spectrum.
\end{Main}

\subsection{Dynamics of cocycles and spectral theory}
The proofs of the preceding  Theorems are based, as is now often the case in the study of 1D quasi-periodic Schr\"odinger operators,  on the study of the so-called family of {\it Schr\"odinger cocycles} associated to the operator. This is the family of skew-product diffeomorphisms $(\a,S_{E-V}):\T\times\R^2\toitself$, $(\a,A)(x,y)=(x+\a,S_{E-V}(x)y)$ where $E\in\R$ and  $S_{W}$ denotes the matrix $\bm W& -1\\ 1&0\em\in SL(2,\R)$. The {\it dynamical} behavior of the family of  cocycles $(\a,S_{E-V})$, $E\in\R$ is closely related to the {\it spectral} properties of the family of operators $H_{V,\a,x}$, $x\in\T$. One usually distinguishes two different kinds of dynamical behaviors:

\smallskip\noindent{\it Reducible or almost-reducible dynamics}: the cocycle $(\a,S_{E-V})$ is $C^k$-{\it reducible}, $k\in\N\cup\{\infty,\omega\}$ if there exists $B:\R/2\Z\to SL(2,\R)$ of class $C^k$ such that $B(\cdot+\a)S_{E-V}(\cdot)B(\cdot)^{-1}$ is a constant $C\in SL(2,\R)$ or equivalently if $(0,B)\circ (\a,S_{E-V})\circ (0,B)^{-1}=(\a,C)$.  This notion is particularly useful when $C$ is an elliptic matrix because then the iterates $(\a,S_{E-V})^n=:(n\a,S_{E-V}^{(n)})$ are bounded in the sense that the products ($n\geq 1$)
$$S_{E-V}^{(n)}(\cdot)=S_{E-V}(\cdot+(n-1)\a)\cdots S_{E-V}(\cdot)$$
are uniformly $C^k$- bounded: indeed, $S_{E-V}^{(n)}(\cdot)=B(\cdot+n\alpha)^{-1}C^nB(\cdot)$. An equally important notion is that of {\it almost-reducibility}: the cocycle $(\a,S_{E-V})$ is {\it almost-reducible} if there exists a sequence $B_{n}:\R/2\Z\to SL(2,\R)$ of class $C^k$ such that $(0,B_{n})\circ (\a,S_{E-V})\circ (0,B_{n})^{-1}$ converges to a constant (elliptic or parabolic) cocycle\footnote{When $k=\omega$ we require the convergence to hold on some fixed analyticity band $(\R+i[-h,h])/\Z$ (where $i=\sqrt{-1}$).}. In many cases, almost-reducibility of  $(\a,S_{E-V})$  implies  slow growth of the iterates $(\a,S_{E-V})^n$. In turn, the boundedness (or the slow growth) of the iterates of the cocycles $(\a,S_{E-V})$ provides useful spectral informations: if this occurs at a point $E$ in the spectrum then, assuming $\a$ diophantine, there exists an interval containing $E$ on which the spectrum of  $H_{V,\a,x}$ is ac (for all $x\in\T$); {\it cf.} Theorem \ref{critac}.

\smallskip\noindent{\it Non-Uniform Hyperbolicity (NUH)}: a cocycle $(\a,S_{E-V})$ is Non-Uniformly Hyperbolic if its (upper) {\it Lyapunov exponent}
$$LE(\a,S_{E-V})=\lim_{n\to\infty}\int_{\T}\frac{1}{n}\ln\|S_{E-V}^{(n)}(x)\|dx
$$
is positive and if it is at the same time not uniformly-hyperbolic or equivalently if $E\in \Sigma_{V,\a}$.
The Theorem of Bourgain and Goldstein \cite{BG} is in fact the statement that if for all $E$ and all $\a$ the Lyapunov exponent of $(\a,S_{E-V})$ is positive then  the operator  $H_{V,\a,0}$  has pp spectrum for Lebesgue a.e. $\a$. The same statement holds if the preceding conditions are satisfied for all $E$ and all irrational $\a$ in some nonempty open intervals. Proving for a given (analytic) cocycle that $LE(\a,S_{E-V})$ is positive is usually done {\it via} {\it Herman's subharmonicity trick} \cite{He}. \footnote{Another method is provided by \cite{Young} with less regular potentials but with diophantine conditions on the frequency and weaker conclusion.}

\smallskip To produce mixed spectrum in Theorem \ref{MainA} one idea is thus to produce ranges of the energy parameter $E$ for which $(\a, S_{E-V})$ has a reducible (or almost-reducible) behavior and other ranges for which it has a NUH behavior.
It turns out that if $V$ is in the class $\cP^\infty$ with $L(V)<1/q$,  this is the case for the family of cocycles $(\a,S_{E-V})^q$ at least if {\it $\a$ is very close to a rational of the form $p/q$.} Indeed, in this case  $(\a,S_{E-V})^q$ will be of the form $(q\a,S_{E-V}^{(q)})$ where for $E$ in some range of parameters one has: (i) $\max_{x\in\T}|\tr S_{E-V}^{(q)}|<2$; and for some other range of parameters $E$: (ii) the image of $\T$ by $\tr S_{E-V}^{(q)} $ contains strictly $[-2,2]$. This leads to the study of cocycles that we shall call {\it totally elliptic} in case (i)  and {\it regular of mixed type} in case (ii). The relevant theorem attached to case (i) is Theorem \ref{theo:mainac} proved in Section \ref{sec:2}. By using a method of algebraic conjugation (also known as the {\it Cheap Trick}, \cite{FK}, \cite{AFK}) we can reduce the analysis of an {\it a priori} non-perturbative situation   to the more classical one of the perturbation of a constant cocycle (Eliasson's Theorem). The theorem that allows us to prove positivity of the Lyapunov exponents in case (ii) is Theorem \ref{theo:5.1} of Section \ref{sec:4}. The method, though it relies on subharmonicity, is different in spirit from Herman's method and closer to the acceleration theory developed by Avila in \cite{A-Acta} even if we don't need the full strength of this theory. 

The dynamical analysis in the proof of Theorem \ref{MainC} uses results of Section \ref{sec:2} and Herman's subharmonicity trick. 

As a final comment, we mention that the current method we use does not allow us to produce many alternating ac and pp components. The  analysis of a more general situation than (ii), namely that of cocyles of mixed type  which are not regular (if the cocycle is not regular this only means that the image of $\T$ by $\tr S_{E-V}^{(q)} $ contains 2 or $-2$) is then necessary. We defer this analysis to another paper.

\subsection{On the genesis of the paper} 

The results in the present paper are based on questions that arose 
when we numerically investigated the Lyapunov exponent of the Schr\"odinger cocycle $(\alpha,S_{E-V_{K,\lambda}})$ with the 
"peaky" potential 
$$
V_{K,\lambda}(x)=\frac{K}{1+\lambda \sin^2(\pi x)}.
$$ 
When fixing $K>5$ and taking $\lambda$ large, the simulations indicated that the potential $V_{K,\lambda}$ behaves both as a "large" potential 
(in the region $|E|>2+\varepsilon$) 
and as a "small" potential (on regions inside the interval $|E|<2|$). 
More precisely, we noted that the Lyapunov exponent $LE(\alpha,S_{E-V_{K,\lambda}})$ seems to be positive for $E$ outside $[-2.1,2.1]$; but inside 
$[-2,2]$ it looks like there are values of $E$ for which the Lyapunov exponent vanishes (see Figures (\ref{fig1}) and (\ref{fig2}); 
in Figure (\ref{fig1}) we have also included a plot of the rotation number $\rho_{\a}(E):=\rho(\a,S_{E-V_{K,\l}})$. 
This seemed unexpected. Since the potential 
is not small, we could not explain this latter observation by directly applying reducibility results (such as Eliasson's theorem \cite{El-cmp}). Provided that the 
numerical experiments showed a correct picture, there had to be some resonance phenomena going on. This is where our journey began.

Our Theorem \ref{MainC} explains this phenomenon for $\a$ diophantine very close to $1/2$. It would be nice to have an explanation for the Golden Mean $\a=(\sqrt{5}-1)/2$, a number of constant type. Maybe the renormalization techniques of \cite{AK1} and \cite{AK2} can be of some use to tackle this question.
\begin{figure}%
    \centering
    \subfloat[$E\mapsto LE(\alpha,S_{E-V_{K,\lambda}})$ on the interval $(-3,10)$]{{\includegraphics[width=6cm]{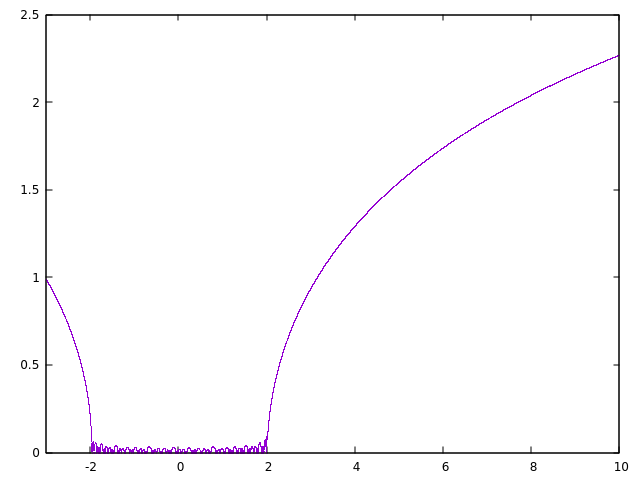} }}%
    \qquad
    \subfloat[$E\mapsto \rho_\alpha(E)$ on the interval $(-3,10)$]{{\includegraphics[width=6cm]{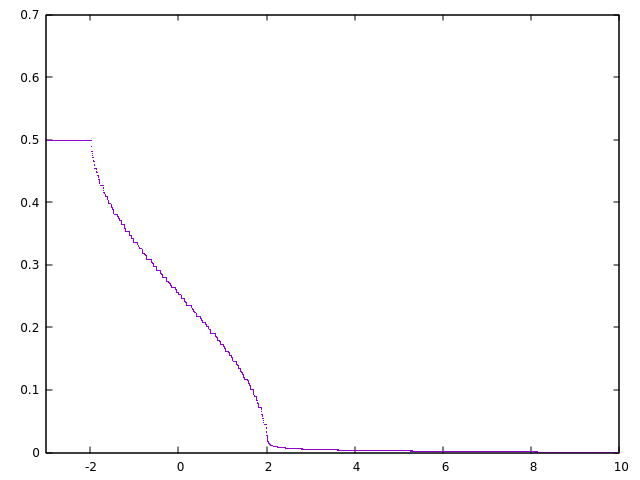} }}%
    \caption{Numerical computations of (A) the Lyapunov exponent $LE(\alpha,S_{E-V_{K,\lambda}})$ 
and (B) the rotation number $\rho_\alpha(E)$ for $\alpha=(\sqrt{5}-1)/2, K=10$ and $\lambda=10000$.}%
    \label{fig1}
\end{figure}
\begin{figure}
\includegraphics[width=10cm]{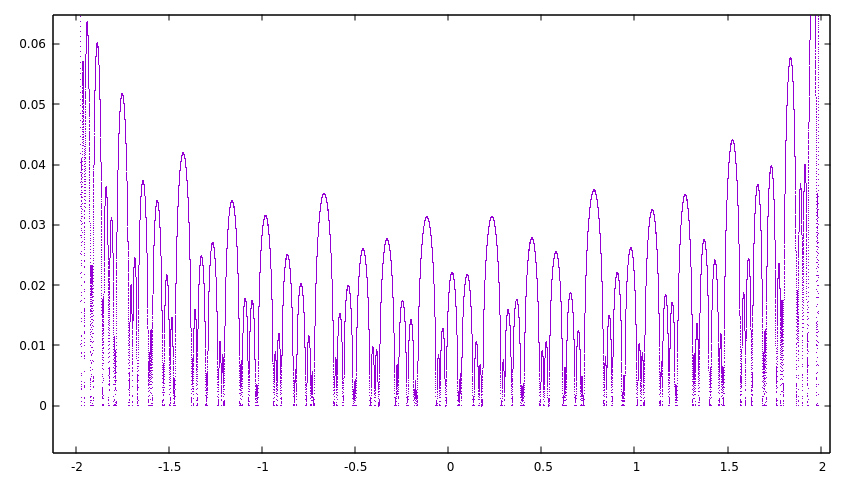}
\caption{Numerical computations of  $E\mapsto LE(\alpha,S_{E-V_{K,\lambda}})$ on $[-2,2]$ for $\alpha=(\sqrt{5}-1)/2, K=10$ and $\lambda=10000$.}\label{fig2}
\end{figure}

\section{Criteria for ac and pp spectrum}
\subsection{Background}
\subsubsection{Norms} If $f:\R/\Z\to\R$ is a $C^\infty$-function we define its $C^k$ norm $\|f\|_{C^k}$ (or for short $\|f\|_{k}$) by
$$\|f\|_{C^k}=\max_{0\leq l\leq k}\sup_{x\in\T}|\pa^l f(x)|.
$$
If $f:\R/\Z\to \R$ is a real-analytic function having a bounded  holomorphic extension $\ti f$ on a complex strip $(\R+i[-h,h])/\Z$ we define 
$$\|f\|_{h}=\sup_{z\in(\R+i[-h,h])/\Z}|\ti f(z)|.
$$
We denote by $C^\omega_{h}(\T)$ the set of all such functions.
\subsubsection{Cocycles}
A $C^k$ ($k\in\N\cup\{\infty,\omega\}$)  {\it quasi-periodic cocycle} is a map $\T^d\times \R^2\toitself$, $(x,y)\mapsto (x+\a,A(x)y)$ where $\a\in\T^d$ and $A:\T^d\to SL(2,\R)$ is  of class $C^k$. If $A$ is a constant map we say that the cocycle $(\a,A)$  is constant and we say that it is respectively  elliptic, unipotent or hyperbolic if $A$ is respectively elliptic ($|\tr(A)|<2$), unipotent ($|\tr(A)|=2$) or hyperbolic ($|\tr(A)|>2$). 
Two cocycles $(\a,A_{i})$, $i=1,2$,  are said to be conjugated if there exists a map  $B:(\R/2\Z)^d\to SL(2,\R)$ of class $C^k$ such that $(0,B)\circ(\a,A_{1})\circ (0,B)^{-1}=(\a,A_{2})$ or equivalently $A_{2}(\cdot)=B(\cdot+\a)A_{1}(\cdot)B(\cdot)^{-1}$.
A cocycle $(\a,A)$ is said to be {\it $C^k$-reducible} if it can be conjugated to a {\it constant} cocycle
and {\it $C^k$-almost reducible}
  if there exists a sequence of maps $B_{n}:(\R/2\Z)^d\to SL(2,\R)$ of class $C^k$ such that the sequence  $B_{n}(\cdot+\a)A(\cdot)B(\cdot)^{-1}$ converges  in $C^k$-topology  \footnote{When $k=\omega$ we require the convergence to hold on some fixed analyticity band $(\R+\sqrt{-1}[-h,h])/\Z$} to a constant elliptic or unipotent  cocycle\footnote{We exclude constant hyperbolic cocycles in this definition just to have nicer statements in the paper.} as $n$ goes to infinity.  

\subsubsection{Lyapunov exponent}  
The {\it (upper) Lyapunov exponent} of a cocycle $(\a,A)$ is the nonnegative  limit 
$$LE(\a,A):=\lim_{n\to\infty}\frac{1}{n}\int_{\T^d}\ln \|A_{n}(x)\|dx$$
where $(n\a,A_{n})=(\a,A)^n$ (equivalently for $n\geq 1$, $A_{n}(\cdot)=A(\cdot+(n-1)\a)\cdots A(\cdot)$).

By Kingman's Ergodic  Theorem, when  the translation $x\mapsto x+\a$ is minimal on $\T^d$  one has for Lebesgue a.e. $x\in\T^d$ 
\be LE(\a,A)=\lim_{n\to\infty}\frac{1}{n}\ln \|A_{n}(x)\|.\label{LE}
\ee
If furthermore $LE(\a,A)>0$ there exists for a.e. $x\in\T^d$ a measurable decomposition 
\be \R^2=E_{s}(x)\oplus E_{u}(x)\label{oseledec}\ee  
such that $A(x)E_{s,u}(x)=E_{s,u}(x+\a)$ and for any $v\in E_{s}(x)$ (resp. $v\in E_{u}(x)$) the limit of $(1/|n|)\ln \|A_{n}(x)v\|$ when $n$ goes to $\infty$ (resp. $-\infty$) is equal to $-LE(\a,A)$.
 
 When $LE(\a,A)>0$  and the convergence in (\ref{LE}) is uniform (the decomposition (\ref{oseledec}) is then continuous) one says that the cocycle $(\a,A)$ is {\it uniformly hyperbolic} (UH). If $LE(\a,A)>0$ and the cocycle $(\a,A)$  is not UH we   say it is {\it non-uniformly hyperbolic} (NUH).

\subsubsection{Rotation number} A continuous cocycle $(\a,A):\T^d\times \R^2\toitself$ naturally defines a {\it projective cocycle} $f_{(\a,A)}:\T^d\times \bS^1\toitself$, $(x,y)\mapsto (x+\a,A(x)y/\|A(x)y\|$). Let $\pi:\R\to\bS^1$ be the projection $y\mapsto (\cos(2\pi y),\sin(2\pi y))$. If the map $A:\T^d\to SL(2,\R)$ is {\it homotopic to the identity}, there exists a lift $F_{(\a,A)}:\T^d\times\R\toitself$, $(x,y)\mapsto (x+\a,g(x,y))$, $g$ being   continuous of the form $g(x,y)=y+\ph(x,y)$ where   $\ph$ is 1-periodic in $y$, such that  $\pi (g(x,y))=A(x)y/\|A(x)y\|$. If $\a$ has rationally independent coordinates  one can prove (\cite{He}) that the following limit
\be r_{(\a,A)}(x,y):=\lim_{n\to\infty}\frac{g(F_{(\a,A)}^n(x,y))-y}{n}\label{rotnumb}
\ee
exists for all $(x,y)$ and is independent of $(x,y)$. This limit is denoted by $\bar \rho(\a,A)$ and its class modulo 1 is denoted by $\rho(\a,A)$ (it is independent of the preceding construction) and is called the {\it fibered rotation number} (for short the rotation number) of the cocycle $(\a,A)$.

We extend the preceding definition when $d=1$ to the case where $\a$ is rational. If  $\a=0$ the limit  $r_{(0,A)}(x,y)$ in equation (\ref{rotnumb}) is defined for any $x\in\T$ and any $y\in \R$, is independent on $y$ but then depends on $x$. We define
$$\bar \rho(0,A)=\int_{\T}r_{(0,A)}(x,y)dx.
$$
When $\a=p/q$ is rational $(p,q)\in\Z\times\N^*$, ($p\wedge q=1$) we define
$$\bar \rho(p/q,A)=\frac{1}{q}\bar \rho((p/q,A)^q).$$
The value of $\rho(\a,A):=\bar\rho(p/q,A) \mod 1$ does not depend on the preceding construction.
 
\subsubsection{Schr\"odinger cocycles} 
A cocycle of the form $(\a,S_{V})$ with $S_{V}=\bm V & -1\\ 1& 0\em$ where $V\in C^k(\T^d,\R)$ is said to be a {\it Schr\"odinger cocycle}. The family $(\a,S_{E-V})$, $E\in\R$ is called the family of Schr\"odinger cocycles  associated to the potential $V$. When $\a$ has rationally independent coordinates, the dynamics of this family of cocycles is intimately related to the spectral properties of the family of  Schr\"odinger operator $H_{V,\a,x}$, $x\in\T^d$. For example $E\notin \Sigma_{V,\a}$ if and only if $(\a,S_{E-V})$ is a uniformly hyperbolic cocycle.  Another example of such a correspondence is the following. Define the probability measure $\nu_{V,\a}:=\int_{\T}\mu_{V,\a,x}dx$. Then for any $E\in\R$, the rotation number of $(\a,S_{E-V})$ is related to the {\it density of states} $\nu_{V,\a}(]-\infty, E])$ by
\be \nu_{V,\a}(]-\infty, E])=1-2\rho(\a,S_{E-V}).\label{nurho}
\ee
 
\subsubsection{Diophantine conditions}
For $d\in\N^*$, $\eta>0,\sigma\geq d+1$ we denote by $DC_{d}(\eta,\sigma)$ the set of $\a\in\T^d$ such that 
$$\forall\ (k,l)\in(\Z^d\setminus\{0\})\times\Z,\ |\<k,\a\>-l|\geq \frac{\eta}{|k|^{\sigma-1}}.$$
For $\sigma>d+1$ and $\eta$ small enough this is a set of positive Lebesgue measure and the set $DC_{d}(\sigma)=\bigcup_{\eta>0}DC_{d}(\eta,\sigma)$ is a set of full Lebesgue measure. An element $\a\in\T^d$ is said to be {\it diophantine} if it is in the set $DC_{d}=\bigcup_{\sigma\geq d-1}DC_{d}(\sigma)$.
 
We say that $\rho\in\R$ is in $DS_{\a}(\kappa)$ if 
$$\forall\ k\in\Z^d\setminus\{0\}, \ \min_{l\in\Z}|2\rho-\<k,\alpha\>-l|\geq \frac{\kappa}{|k|^{d+2}}$$ and we set $DS_{\a}=\cup_{\kappa>0}DS_{\a}(\kappa)$. Elements of $DS_{\a}$ are said to be {\it diophantine with respect to $\a$}. It is easy to check that $DS_{\a}$ is a set of full Lebesgue measure. 
\subsection{Criterium for ac spectrum}
We first recall the following fundamental result by Eliasson:
\begin{theo}[Eliasson, \cite{El-cmp}]\label{theo:eliassoncmp}Let $d\in\N^*$,  $\a \in DC_{d}(\eta,\sigma)$  and $h>0$. There exists $\varepsilon_{0}(\eta,\sigma,h)$ such that if  $V\in C^\omega_{h}(\T^d,\R)$ satisfies $\|V\|_{h}\leq \varepsilon_{0}$  then  the spectral measure $\mu_{V,\alpha,x}$ is ac for any $x\in\T^d$. 
\end{theo}
\begin{rem} Strictly speaking, Eliasson's Theorem is proven in \cite{El-cmp} for  quasi-periodic  1D-Schr\"odinger operators  on the real line (continuous time) with diophantine frequency vectors  and with   small real analytic potentials \footnote{The theorem is also true for any real analytic potential and  for large energies in the spectrum}.  The  analysis of the  1D discrete quasi-periodic Schr\"odinger operators  on $\Z$  with diophantine frequency vector and small real analytic potential is done in \cite{HadjAmor}. 
\end{rem}
The proof of the preceding result is based on a KAM-type inductive procedure that allows for  a very precise description of the dynamical properties of the family of Schr\"odinger cocycles $(\a,S_{E-V})$: in this perturbative regime these cocycles are always uniformly hyperbolic (which is equivalent to $E\notin \Sigma_{V,\a}$) or almost-reducible. The KAM-type inductive scheme of \cite{El-cmp},  \cite{HadjAmor} and the description of the dynamics of the family  of Schr\"odinger cocycles  can be extended to the smooth case ({\it cf.} \cite{FK} and Section \ref{sec:eliassoncinfty}) and 
a close look at Eliasson's proof of \cite{El-cmp} shows that  the spectral result of Theorem \ref{theo:eliassoncmp}  can as well  be extended to the smooth case.
Indeed,   the following more general version of Theorem \ref{theo:eliassoncmp} is valid:
\begin{theo}\label{Eliassongeneral}Let  $d\in\N^*$, $\a\in\T^d$ be diophantine, $V\in C^\infty(\T^d,\R)$ and assume that  there exists $B\in C^\infty((\R/2\Z)^d,SL(2,\R))$ such that for $E$ in some  interval $]\bar E-\delta,\bar E+\delta[$ ($\delta>0$) $B(\cdot+\a)^{-1}S_{E-V}(\cdot) B(\cdot)$ is of the form $A(E-\bar E)e^{F_{E}(\cdot)}$ where $E\mapsto A(E)\in SL(2,\R)$ and $E\mapsto F_{E}(\cdot)\in C^\infty(\T^d,sl(2,\R))$ are  real analytic and  such that 
$$\frac{d}{dE}\rho(\a,A(E))_{|(E=\bar E)}\ne 0
$$
 and
$$\|F_{E}(\cdot)\|_{C^0}=O((E-\bar E)^2).
$$
Then, there exists an nonempty open   interval $I$ containing $\bar E$ on which $\mu_{V,\a,x}$ is (non trivial and) ac for all $x\in\T^d$.
\end{theo} 

The preceding theorem holds true at least on intervals of energies where the cocycle $(\a,S_{E-V})$ is  almost-reducible. 
\begin{theo} [Extension of Eliasson's Theorem]\label{critac}Let $V\in C^\infty(\T^d,\R)$ and $\a$ diophantine. If for some $E_{*}\in\R$ the cocycle $(\a,S_{E_{*}-V})$ is $C^\infty$ almost-reducible then there exists a nonempty open  interval $I_{ac}$  such that for all $x\in\T^d$, the restriction of  the spectral measure $\mu_{V,\alpha,x}$ to $I_{ac}$ is (non-trivial and) ac. 
\end{theo}
\begin{proof}
See the Appendix.
\end{proof}

\subsection{Criterium for pp spectrum}
We recall the following criterium for the existence of pp spectrum in some interval due to Bourgain and Goldstein \cite{BG} (see also the nice survey  \cite{JM}). 
\begin{theo}[Bourgain-Goldstein Theorem]\label{theo:BG}Let $V\in C^\omega(\T,\R)$, $I_{pp}$ and $J$ be two nonempty open intervals. Assume that  for all $(E,\a)\in I_{pp}\times (J\setminus(\Q/\Z))$ one has $LE(\a,S_{E-V})>0$. Then, there exists a set $\mathcal{BG}\subset J$, ${\rm Leb}(J\setminus\mathcal{BG})=0$, such that for all $\a\in\mathcal{BG}$, the  restriction of the spectral measure $\mu_{V,\alpha,0}$ to $I_{pp}$ is pp (and satisfies Anderson localization).
\end{theo}

\section{Totally elliptic cocycles}\label{sec:2}

We say that $A:\T\to SL(2,\R)$ is {\it totally elliptic} if for any value of $x$, $|{\rm tr }A(x)|<2$. This is equivalent to the fact that for any  $x$, $A(x)$ is an elliptic matrix (a matrix in $SL(2,\R)$ which is conjugated to a rotation matrix different from $\pm I$). If $p/q$ is a rational number and $\prod_{k=q-1}^0 A(\cdot+kp/q)$ is totally elliptic, we say that the cocycle $(p/q,A)^q$ is totally elliptic. 

For $\eta>0,\sigma\geq 2$ we recall that $DC_{1}(\eta,\sigma)$ is the set of $\a\in\T$ such that 
$$\forall\ (k,l)\in\Z\times\N^*,\ |\a-\frac{k}{l}|\geq \frac{\eta}{l^{\sigma}}$$
and  we denote by $D_{p/q}(\eta)$ 
$$D_{p/q}(\eta)=]\frac{p}{q}-\eta,\frac{p}{q}+\eta[\cap DC_{1}(\eta^2,3).$$

\begin{lem}\label{lem:posmeas} For any $0<\eta<1/2$, $D_{p/q}(\e)$  is a set of positive Lebesgue measure. 
\end{lem}
\begin{proof}
We have to prove that  the Lebesgue measure of 
$$E:=]\frac{p}{q}-\eta,\frac{p}{q}+\eta[\setminus \bigcup_{(k,l)\in\Z\times\N^*}]\frac{k}{l}-\frac{\eta^2}{l^3},\frac{k}{l}+\frac{\eta^2}{l^3}[
$$
is positive. This measure can be bounded by below by 
$$2\eta-\sum_{l=1}^\infty l\frac{2\eta^2}{l^3}\geq 2\eta(1-2\eta).
$$
\end{proof}

This section is dedicated to the proof of the following theorem.
\begin{theo}\label{theo:mainac}There exists $s_{0}\in\N$ for which the following is true. Let $p/q$ be a rational number, $I\subset \R$ a compact  interval with nonempty interior and a map $I\times \T\to SL(2,\R)$, $(E,x)\mapsto A_{E}(x)$ which is  homotopic to the identity, $C^1$ in $E$ and smooth in $x$. Assume that 
\begin{itemize}
\item for any $E\in I$,  $((p/q),A_{E})^q$ is totally elliptic;
\item the map $I\to\R$, $E\mapsto \rho(p/q,A_{E})$ is not constant.
\end{itemize}
Then, there exists $\e_{0}(\max_{E\in I}\|A_{E}\|_{C^{s_{0}}},q)>0$ such that for any $0<\e\leq \e_{0}$ and  any  $\a\in D_{p/q}(\e)$ there exists a positive Lebesgue measure set  $\cE_{\a}\subset I$ such that for any $E\in\cE_{\a}$, the cocycle $(\a,A_{E})$ is reducible (conjugated to a constant elliptic cocycle). 
 \end{theo}
\subsection{Periodic approximations}
Let $R_{\ph}$ be the matrix $$R_{\ph}:=\bm \cos \ph & -\sin\ph\\ \sin \ph&\cos\ph\em.$$
\begin{prop}[Periodic approximation]\label{prop:perapprox}If $A:\T\to SL(2,\R)$ is  homotopic to the identity, smooth and satisfies ${\rm tr}|\prod_{k=q-1}^0 A(\cdot+kp/q)(\cdot)|<2$, then there exist smooth maps  $B:\T\to SL(2,\R)$ and  $\ph:\T\to\R$  such that  
$B(\cdot+\frac{p}{q})^{-1}A(\cdot)B(\cdot)=R_{\ph(\cdot)}$.
\end{prop}
\begin{proof}We notice that if we denote $A^{(q)}(\cdot)=A(\cdot+(q-1)p/q)\cdots A(\cdot)$ then 
\be A^{(q)}(\cdot+\frac{p}{q})=A(\cdot)A^{(q)}(\cdot)A(\cdot)^{-1}.\label{eq1}\ee 
By assumption, for all $x\in\T$ one has $|{\rm tr}A^{(q)}(x)|<2$ so by Lemma \ref{app:0} there exist  smooth maps $B:\T\to SL(2,\R)$, $a:\R\to ]0,\pi[$ such that 
$$\forall x\in\T, \ B(\cdot)^{-1}A^{(q)}(x)B(x)=R_{a(x)}$$
and from (\ref{eq1}) we get 
$$R_{a(\cdot+p/q)}=\biggl(B(\cdot+\frac{p}{q})^{-1}A(\cdot)B(x)\biggr)R_{a(\cdot)}\biggl(B(\cdot+\frac{p}{q})^{-1}A(\cdot)B(x)\biggr)^{-1}.$$
But since for any $x\in\T$, $0<a(x)+a(x+p/q)<2\pi$ and $A$ is homotopic to the identiy, Lemma \ref{app:1} applies and we get the existence of a smooth $\ph:\T\to \R$ such that 
$$B(\cdot+\frac{p}{q})^{-1}A(\cdot)B(x)=R_{\ph(\cdot)}.$$
\end{proof}

\subsection{Quantitative Eliasson Theorem}\label{sec:eliassoncinfty}

Let $\kappa>0$. We say that $\rho\in\R$ is in $DS_{\a}(\kappa)$ if 
$$\forall\ k\in\Z^*, \ \min_{l\in\Z}|2\rho-k\alpha-l|\geq \frac{\kappa}{|k|^2}$$ and we set $DS_{\a}=\cup_{\kappa>0}DS_{\a}(\kappa)$. Elements of $DS_{\a}$ are said to be {\it diophantine with respect to $\a$}. It is easy to check that $DS_{\a}$ is a set of full Lebesgue measure. 

We recall the following quantitative  version of Eliasson's theorem proved in \cite{FK}.

\begin{theo} \label{theo:eliasson}There exist two  constants $C_1,C_2>0$ such that the following holds. Let $\hat{A} \in {\rm SL}(2,\R)$ and  $\sigma>0$, be fixed.  Then there exists
$\epsilon(\sigma,\hat{A})>0$   such that if a cocycle $(\a,A) \in  \R \times C^\infty(\T, SL(2,\R))$ satisfies 
\begin{itemize}

\item[(1)]  $\a \in {\rm DC}_{1}(\gamma,\sigma)$; 

\item[(2)] $\rho(\a,A)$ is Diophantine with respect to $\a$;

\item[(3)]  ${\|A - \hat{A}\|}_{C^{0}} \leq \gamma^{d_0}\epsilon$ and ${\|A - \hat{A}\|}_{C^{s_0}} \leq 1$ , where $d_0=C_1\sigma$, $s_0=[C_2 \sigma]$

\end{itemize}
then $(\a,A)$ is $C^\infty$ reducible (conjugated to a constant elliptic cocycle if $\rho(\a,A)\ne 0$). 
\end{theo}

\bigskip\noindent{\bf Notation:} Till the end of Section \ref{sec:2} we denote the $C^s$-norm $\|\cdot\|_{C^s}$ by $\|\cdot\|_{s}$.

\subsection{The Cheap Trick}

{We denote by $s_{0}=[3C_{2}]$ where $C_{2}$ is the constant of Theorem \ref{theo:eliasson}.}
The aim of this section is to prove the following proposition. 

\begin{prop}[Cheap trick] \label{prop:main}Let $p/q$ be a rational number and $A:\T\to SL(2,\R)$ such that $2-\d:=\max_{\T}|{\rm tr}(\prod_{k=q-1}^0 A(\cdot+kp/q))|<2$. Then, there exists  $\eta_{0}(\|A\|_{C^{s_{0}}},q,\d)>0$ such that for any $0<\eta\leq \eta_{0}$, and any $\a\in D_{p/q}(\e)$,  there exists $B\in C^\infty(\R/2\Z,SL(2,\R))$  such that $ B(\cdot+\a)^{-1}A(\cdot)B(\cdot)$ is a constant matrix in $SO(2,\R)$  
provided   $\rho(\a,A)\in DS_{\a}\setminus\{0\}$.
\end{prop}
\begin{proof}
The proof of the proposition is a consequence of Corollary \ref{cor:main} below which itself is deduced from the following two lemmas.
\begin{lem}[Inductive step]\label{lem:indstep}Let $\ph:\T\to\R$ be smooth and assume that $\psi(\cdot):=\sum_{k=0}^{q-1}\ph(\cdot+kp/q)$ satisfies 
$\d:= \min_{x\in\R,k\in\Z}|\psi(x)-k\pi|>0$. 
Then, given $r\in\N$,  there exists $\nu_{0}>0$ such that for  any $\a\in \T$ and any $F\in C^\infty(\T,sl(2,\R))$ such that $\|F\|_{r}\leq \nu_{0}$, the following is true: there exist $\ti \ph\in C^\infty(\T,\R)$, $\ti F,Y\in C^\infty(\T,sl(2,\R))$, such that 
$$(0,e^{-Y(\cdot)})\circ (\a,e^{F(\cdot)}R_{\ph(\cdot)})\circ (0,e^{Y(\cdot)})=(\a,e^{\ti F(\cdot)}R_{\ti\ph(\cdot)})$$
with
\begin{align}&\|\ti\ph-\ph\|_{r}\leq C^{(3)}(r,\|\ph\|_{r},\d)\|F\|_{r} \\ &\|Y\|_{r}\leq C^{(3)}(r,\|\ph\|_{r},\d)\|F\|_{r} ,\\  &\|\ti F\|_{r-1}\leq C^{(3)}(r,\|\ph\|_{r},\d)\|F\|_{r}  |\a-\frac{p}{q}|. \end{align}
where $C^{(3)}(r,s,t)$ is a positive non-decreasing  function of   $r$, $s$ and $t^{-1}$.
\end{lem}
\begin{proof}We set $\nu:=\|F\|_{r}$. Let us consider the cocycle $(p/q,e^{F(\cdot)}R_{\ph(\cdot)})$. Its $q$-th iterate $(p/q,e^{F(\cdot)}R_{\ph(\cdot)})^q$ is of the form $(0,e^{G(\cdot)}R_{\psi(\cdot)})$ where $\psi(\cdot)=\sum_{k=0}^{q-1}\ph(\cdot+kp/q)$ and $G:\T\to\ sl(2,\R)$ satisfies  (as can be seen from Lemma 9 of \cite{FK}): for any $0\leq j\leq r$ 
$$\|G\|_{j} \leq K_{r}q^{j+1}(1+\nu)^{q-2}(1+\|\ph\|_{j})\nu\leq C_{r}(q,\|\ph\|_{j})\nu.$$
Using Lemma \ref{app:2} we deduce that there exists $\ti\psi$ and  such that $e^{G(\cdot)}R_{\psi(\cdot)}=e^{Y(\cdot)}R_{\ti\psi(\cdot)}e^{-Y(\cdot)}$  or equivalently $(p/q,e^FR_{\ph})^q=(0,e^{Y})\circ (0,R_{\ti \psi})\circ (0,e^{-Y})$ where 
$$\|Y\|_{r}\leq C^{(1)}(r,\|\ph\|_{r},\d)\nu,\qquad (\text{and}\  \|\ti\psi-\psi\|_{r}\leq C^{(1)}(r,\|\ph\|_{r},\d)\nu).$$
From Proposition \ref{prop:perapprox} we see that there exists  a smooth $\ti \ph:\T\to\R$ such that 
$$(0,e^{-Y})\circ (p/q,e^FR_{\ph})\circ (0,e^{Y})=(p/q,R_{\ti \ph}).$$
From the previous  identity
\be e^{-Y(\cdot+p/q)}e^{F(\cdot)}R_{\ph(\cdot)}e^{Y(\cdot)}=R_{\ti\ph(\cdot)}\label{eq:2}\ee
 and from  the inequalities $\|e^{\pm Y}-I\|_{r}\leq C_{r}(\|Y\|_{0})\|Y\|_{r}$ and $\|e^{F}-I\|_{r}\leq C_{r}(\|F\|_{0})\|F\|_{r}$ we deduce that 
\begin{align*}\|\ti \ph-\ph\|_{r}&\leq \max(C_{r}(\|Y\|_{0}),\|F\|_{0})(\|Y\|_{r}+\|F\|_{r})\\
&\leq C^{(2)}(r,\|\ph\|_{r},\d)\nu\end{align*}
Now, from equation (\ref{eq:2}) we get 
\begin{align} e^{-Y(\cdot+\a)}e^{F(\cdot)}R_{\ph(\cdot)}e^{Y(\cdot)}&=e^{-Y(\cdot+\a)}e^{Y(\cdot+p/q)}R_{\ti\ph(\cdot)}\label{eq:3}\\ 
&:=e^{\ti F(\cdot)}R_{\ti\ph(\cdot)}\end{align}
whith
$$\|\ti F\|_{r-1}\leq C_{r}(\|Y\|_{0})\|Y\|_{r}|\a-\frac{p}{q}|.$$
\end{proof}

\begin{lem}\label{lem:2.6}Under the assumptions of the preceding lemma, given $r\in\N$,  there exists $\nu_{1}>0$ such that for any $\a\in \T$ and any $F\in C^\infty(\T,sl(2,\R))$ such that $\|F\|_{r}\leq \nu_{1}$ the following is true: for any $0\leq j\leq r$, there exist $\ph_{j}\in C^\infty(\T,\R)$, $ F_{j},Z_{j}\in C^\infty(\T,sl(2,\R))$, such that 
$$(0,e^{-Z_{j}(\cdot)}) \circ (\a,e^{F(\cdot)}R_{\ph(\cdot)})\circ (0,e^{Z_{j}(\cdot)})=(\a,e^{ F_{j}(\cdot)}R_{\ph_{j}(\cdot)})$$
with
\begin{align}&\|\ph_{j}-\ph\|_{r-j+1}\leq C^{(3)}(r,\|\ph\|_{r},\d)\|F\|_{r} \\ &\|Z_{j}\|_{r-j+1}\leq C^{(3)}(r,\|\ph\|_{r},\d)\|F\|_{r} ,\\  &\|F_{j}\|_{r-j}\leq C^{(3)}(r,\|\ph\|_{r},\d)\|F\|_{r}  |\a-\frac{p}{q}|^j. \end{align}
where $C^{(3)}(r,s,t)$ is a positive non-decreasing  function of   $r$, $s$ and $t^{-1}$.
\end{lem}
\begin{proof}Just iterate the previous Lemma \ref{lem:indstep} and check that the inductively defined sequence $F_{j}$ satisfies $\|F_{j}\|_{r-j}\leq \nu_{0}$ (the $\nu_{0}$ of Lemma \ref{lem:indstep}).
\end{proof}
\begin{cor}\label{cor:main}Let $p/q$ be a rational number,  $s,m\in\N$, and $A:\T\to SL(2,\R)$ such that $2-\d:=\max_{\T}|{\rm tr}(\prod_{k=q-1}^0 A(\cdot+kp/q))|<2$. Then, there exist $\e_{0}(s,m,\|A\|_{s+m+7}, q,\d)>0$ such that for any $0<\e\leq \e_{0}$, and any $\a\in D_{p/q}(\e)$  there exist $B\in C^\infty(\T,SL(2,\R))$ and $A_{0}\in SO(2,\R)$ such that 
\be \|B(\cdot+\a)^{-1}A(\cdot)B(\cdot)-A_{0}\|_{s}\leq \e^{m}.\ee 
and $$\|B\|_{s}\leq \e^{-3}.$$
\end{cor}
\begin{proof}Let us define $r=s+m+7$ and $j=m+2$. By Proposition \ref{prop:perapprox} there exist $B_{1}\in C^\infty(\T,SL(2,\R))$ and $\ph  \in C^\infty (\T,\R)$ such that $B_{1}(\cdot+p/q)^{-1}A(\cdot)B_{1}(\cdot)=R_{\ph(\cdot)}$ and thus $B_{1}(\cdot+\a)^{-1}A(\cdot)B_{1}(\cdot)=e^{F(\cdot)}R_{\ph(\cdot)}$ with $\|F\|_{r-1}\leq C_{4}\times |\a-(p/q)|$ where the constant  $C_{4}$ depends  indeed on $\|B_{1}\|_{r}$, $r$ and $\d$. Since  $|\cos(\sum_{k=0}^{q-1}\ph(\cdot+kp/q))|=|{\rm tr}(\prod_{k=q-1}^0 A(\cdot+kp/q))|$,  we see that $\min_{\T,l\in\Z}|(\sum_{k=0}^{q-1}\ph(\cdot+k\a))-l\pi|\geq\d$ and we can apply Lemma  \ref{lem:2.6} to $(\a,e^{F}R_{\ph})$ (with $r-1$ in place of $r$): if $C_{4}\e_{0}<\nu_{1}$ then $\|F\|_{r-1}\leq C_{4}|\a-(p/q)|\leq C_{4}\e\leq \nu_{1}$ and  there exist $\ph_{j}\in C^\infty(\T,\R)$, $ F_{j},Z_{j}\in C^\infty(\T,sl(2,\R))$, such that 
$$(0,e^{-Z_{j}(\cdot)}) \circ (\a,e^{F(\cdot)}R_{\ph(\cdot)})\circ (0,e^{Z_{j}(\cdot)})=(\a,e^{ F_{j}(\cdot)}R_{\ph_{j}(\cdot)})$$
with
\begin{align}&\|\ph_{j}-\ph\|_{r-j}\leq C^{(3)}(r,\|\ph\|_{r},\d)C_{4}\e \label{eq:2.14}\\ &\|Z_{j}\|_{r-j}\leq C^{(3)}(r,\|\ph\|_{r},\d)C_{4}\e ,\\  &\|F_{j}\|_{r-j-1}\leq C^{(3)}(r,\|\ph\|_{r},\d)(C_{4}\e)^{j+1} . \end{align}
From (\ref{eq:2.14}) we see that $\|\ph_{j}\|_{r-j}\leq 2 C^{(3)}(r,\|\ph\|_{r},\d)$ provided $C_{4}\e\leq 1$ and since $\a\in D_{p/q}(\e)\subset DC(\eta^2,3)$, there exist $\th(\cdot)\in C^\infty(\T,\R)$ such that the following cohomological equation is satisfied
$$\ph_{j}(\cdot)=\th(\cdot+\a)-\th(\cdot)+\int_{\T}\ph_{j}dx$$ with the estimate 
$$\|\th\|_{r-j-4}\leq \e^{-2} \|\ph_{j}\|_{r-j}\leq  2 \e^{-2} C^{(3)}(r,\|\ph\|_{r},\d).$$
As a consequence,  if we define $B_{j}(\cdot)=e^{Z_{j}(\cdot)}R_{\th(\cdot)}$ and  $\th_{0}:=\int_{\T}\ph_{j}(x)dx$
$$(0,B_{j}(\cdot))^{-1} \circ (\a,e^{F(\cdot)}R_{\ph(\cdot)})\circ (0,B_{j}(\cdot))=(\a,e^{ \ti F_{j}(\cdot)}R_{\th_{0}})$$
with 
$$\|\ti F_j\|_{r-j-4}\leq C^{(4)}(r,\|\ph\|_{r},\d)\e^{j-1}$$ 
and 
$$\|B_{j}\|_{r-j-4}\leq C^{(4)}(r,\|\ph\|_{r},\d)\e^{-2}$$ 
which is the desired conclusion if $\e$ is small enough.
\end{proof}

We can now prove Proposition  \ref{prop:main}.  Using  the notations of Theorem \ref{theo:eliasson} we  choose $\sigma=3$,  $s=s_{0}=[C_{2}\sigma]=[3C_{2}]$, $d_{0}=C_{1}\sigma=3C_{1}$, $m=d_{0}+1$, and apply successively   Corollary \ref{cor:main} and  Theorem \ref{theo:eliasson}.

\end{proof}

\subsection{Variation of the rotation number}\label{sec:varrotnub}
The following lemmas are proved in the Appendix.

\begin{lem}\label{lem:var1}If $I\subset \R$  is an interval and   $\cA:I\times \T\to SL(2,\R)$, $(E,x)\mapsto A_{E}(x)$ is continuous  and homotopic to the identity, then the map $\T\times I\to\R$, $(\a,E)\mapsto \rho(\a,A_{E})$ is continuous.
\end{lem}

\begin{lem}\label{lem:var2}Assume that the map $\cA$ is Lipschitz with respect to $E$ (uniformly in $x$). If there exists $E_{0}\in I$ for which $(\a,A_{E_{0}})$ is  $C^0$-conjugated to a constant elliptic cocycle, then there exists a constant $C_{E_{0}}$ such that  for any $E\in I$
$$|\rho(\a,A_{E})-\rho(\a,A_{E_{0}})\leq C_{E_{0}}|E-E_{0}|.$$
\end{lem}

\subsection{Proof of Theorem \ref{theo:mainac}}\label{sec:3.5}
By assumption, there exists a non empty open interval $J$ such that $J\subset \{\rho(p/q,A_{E}), E\in I\}$. From Lemma \ref{lem:var1} there exists some $\e_{1}$ such that for any $\a\in ]p/q-2\e_{1},p/q+2\e_{1}[$, $J\subset \{\rho(\a,A_{E}), E\in I\}$. 
From Proposition \ref{prop:main} we know that  for any $0<\e<\min(\e_{0},\e_{1})$ and any $\a\in D_{p/q}(\e)$  the cocycle $(\a,A_{E})$ is $C^\infty$-conjugated to a constant elliptic cocycle provided $\rho(\a,A_{E})\in DS_{\a}(\kappa)\setminus\{0\}$  for some $\kappa>0$; choose such an $\a$. Since $\cup_{\kappa>0}DS_{\a}(\kappa)$ has full Lebesgue measure, there exists $\kappa$ for which $M:=DS_{\a}(\kappa)\cap J$ has positive Lebesgue measure. If $f:I\to \R$ is defined by $f(E)=\rho(\a,A_{E})$, Lemma \ref{lem:var2} tells us that for any $E\in f^{-1}(M)$ there exists a constant $C_{E}$ such that for any $E'\in I$,  $|f(E')-f(E)|\leq C_{E}|E'-E|$. We apply to this situation the following lemma:

\begin{lem}\label{lem:3.9}
Let $I,J$ be two intervals of $\R$ and  $f:I\to J$ be  a continuous map. Assume  that there exists a set $M\subset J$, of positive Lebesgue measure, such that
for every $x\in f^{-1}(M)$ there is a constant $C_x>0$ such that $|f(x)-f(x')|\leq C_x|x-x'|$ for all $x'$ (close to $x$).
Then $f^{-1}(M)$ has positive measure.
\end{lem}
\begin{proof}
The proof goes by contradiction. Assume that $f^{-1}(M)$ has zero measure. Let
$$
K_n=\{x\in f^{-1}(M): C_x\leq n\}, \ n\in \mathbb{Z}.
$$
Then we have $K_n\subset K_{n+1}$, and $\bigcup_{n\geq 1}K_n=f^{-1}(M)$. 
By assumption we also have that each $K_n$ is of measure zero.

For any fixed $n\geq 1$, we can cover the zero measure set $K_n$ by smaller and smaller intervals 
(on $K_n$ we have a uniform upper bound on the Lipschitz constant of $f$) and thus show that the measure of $f(K_n)$ must be zero.

Thus it follows that the set $\bigcup_{n\geq 1}f(K_n)$ has zero measure. 
Since we have $$M=f\left(\bigcup_{n\geq 1}K_n\right)=\bigcup_{n\geq 1}f(K_n),$$
we hence get the contradiction.
\end{proof}
As a consequence, the set of $E\in I$ for which $(\a,A_{E})$ is conjugated to a constant elliptic matrix is of positive Lebesgue measure.\hfill$\Box$

\section{Regular cocycles of mixed type}\label{sec:4}
We denote by $\Omega_{h}=(\R+i[-h,h])/\Z$ and $\Omega^+_{h}=(\R+i(0,h])/\Z$. 
\subsection{Regular cocycles}
We have already defined the notion of a  {\it totally elliptic} map $A:\T\to SL(2,\R)$: for any $x\in\T$, $A(x)$ is elliptic. Changing ``elliptic'' to ``hyperbolic'' would lead to the notion of a {\it totally hyperbolic} map. We say that a map $A:\T\to SL(2,\R)$  is of {\it mixed type} if there exists $x\in\T$ for which $|{\rm tr}A(x)| =2$.  Similarly, if $A:\Omega_{h}\to SL(2,\C)$ is holomorphic and  real on $\Omega_{0}$ we say it is of mixed type if its restriction to $\T$ is.  Furthermore, we say that $A:\Omega_{h}\to SL(2,\C)$  is {\it regular} if there exists $0<h'\leq h$ and a holomorphic map $\l_{A}:\Omega^+_{h'}:= (\R+i(0,h'])/\Z\to\C$, such that for any $z\in\Omega_{h'}^+$, $|\l_{A}(z)|>1$ and $\{\l_{A}(z),\l_{A}(z)^{-1}\}$ are the eigenvalues of $A(z)$. In that case the spectral radius, ${r_{\spec}}(A)$,  of $A$ is equal to $|\l_{A}(z)|$. Notice that $h':=h'(A,h)$ can be chosen to depend continuously on $A\in C^\omega_{h}(\T,SL(2,\R))$ regular.

We say that a cocycle of the form $(0,A)$ is regular (resp. has mixed type) if the same is true for $A$.

Here is a criterium for regularity.
\begin{lem}\label{lem:reg}Let $A\in C^\omega_{h}(\T,SL(2,\R))$. If for any $x\in\T$ such that $|{\rm tr}(A(x))|\leq 2$ one has 
$|\frac{d}{dx}{\rm tr A}(x)|> 0$
 then $A$ is regular.
\end{lem}
\begin{proof} Let  us fix $\e>0$ such that $|\frac{d}{dx}{\rm tr A}(x)|\geq \delta>0$ for $|{\rm tr}(A(x))|\leq 2+\e$. We just have to prove that on some $\Omega_{h'}^+$, $h'>0$, one can follow the eigenvalues of $A$ in a holomorphic way and it is enough to prove that for any $z\in\Omega_{h'}^+$ $A(z)$ has no eigenvalue  lying on the unit circle. If $h'$ is small enough, this is clear on $\{x+iy:|{\rm tr}(A)(x)|>2+\e, 0<y\leq h'\}$. On the other hand, if $x+iy\in \{x+iy: |{\rm tr}(A)(x)|\leq 2+\e, 0<y\leq h'\}$ then ${\rm tr}A(x+iy)={\rm tr}A(x)+\frac{d}{dx}{\rm tr}A(x)\cdot(iy)+O(y^2)$. But, if one  eigenvalue of $A(x+iy)$ is on the unit circle ${\rm tr}A(x+iy)$ is real and in $[-2,2]$; thus $0=\Im \frac{d}{dx}{\rm tr}A(x)\cdot(iy)+O(y^2)=\frac{d}{dx}{\rm tr}A(x)y+O(y^2)$ (since $\frac{d}{dx}{\rm tr}A(x)$ is also real). But for $0<y<h'$, $h'$ small enough this is impossible since $|\frac{d}{dx}{\rm tr A}(x)|\geq \delta>0$. 
\end{proof}

The main result is the following
\begin{theo}\label{theo:5.1}Let $(p,q)\in\Z\times\N^*$, $A\in C^\omega_{h}(\T,SL(2,\R))$ and assume that $(p/q,A)^q$ is  regular and of mixed type. Then,
 there exists $\eta_{2}:=\eta_{2}(A,h,q)>0$, continuous with respect to $A\in C^\omega_{h}(\T,SL(2,\R))$,   such that for any $\a$ irrational satisfying $0<|\a-(p/q)|<\eta_{2}$  one has 
$$LE(\a,A)>0.$$

\end{theo}

\subsection{Complex extensions}Let $\a\in(\R-\Q)/\Z$ and $A:\T\to SL(2,\R)$ a real analytic map admitting a holomorphic extension on a complex strip $\Omega_{h}=(\R+[-i h,ih])/\Z$ ($h>0$). For $|\nu|\leq h$ define $A_{\nu}(\cdot)=A(\cdot+i\nu):\T\to SL(2,\C)$. As usual $LE(\a,A_{\nu})$ is the limit of $\int_{\T}(1/n)\log\|\prod_{k=n-1}^0A(x+i\nu+k\a)\|dx$.  Notice that when $\a=p/q$ is rational this limit coincide with  $(1/q)\int_{\T}\log r_{\spec}(\prod_{k=q-1}^0A(x+i\nu+kp/q))dx$.We shall need the following fact:

\begin{lem}\label{lem:cvx}For any $\a\in \T$, the map $L_{\a}:[-h,h]\to \R_{+}$, $\nu\mapsto LE(\a,A_{\nu})$ is an even  convex function.  
\end{lem}
\begin{proof}For $z\in\Omega_{h}$, $n\in\N$ the map $M_{n}:z\mapsto \frac{1}{n}\log \|\prod_{k=n-1}^0A(z+k\a)\|$ is subharmonic, hence the map $L_{n}(\a,\cdot):\nu\mapsto \int_{\T}M_{n}(\cdot+i\nu)$ is convex as well as $\nu\mapsto \lim_{n\to\infty} M_{n}(\nu)$. It is even because $A$ is real on the real axis. 
\end{proof}

Though we will not need it, we mention that a much stronger (and more difficult to prove) statement is the following theorem due to A. Avila \cite{A-Acta}. 
\begin{theo}[Avila]The map $L_{\a}:[-h,h]\to \R_{+}$, $\nu\mapsto LE(\a,A_{\nu})$ is an even  piecewise linear convex function the slopes of which are integer multiples of $2\pi$. Moreover, if $\nu\in (0,h)$ has an open neighborhood on which $L_{\a}(\cdot)$ is linear, the cocycle $(\a,A_{\nu}(\cdot))$ is uniformly hyperbolic. The slope of $L_{\a}$ at $0^+$ is called the acceleration.
\end{theo}

\subsection{Consequences of regularity}
When $A$ is regular and $\a=0$ one can improve the result of Lemma \ref{lem:cvx}.
\begin{lem}\label{lem:reglem}If $A\in C^\omega_{h}(\T,SL(2,\R))$ is regular, the map $L_{0}$ restricted to $[0,h']$ is  a non-negative, non decreasing affine function.
\end{lem}
\begin{proof}Indeed,  in that case   for any $0<\nu\leq h'$, $L_{0}(\nu)= \int_{\T}\log|\l_{A}(x+i\nu)|dx$. Since $\l_{A}(\cdot)$ is holomorphic on $\Omega^+_{h'}$ and non zero, $\log|\l_{A}(\cdot)|$ is harmonic on $\Omega^+_{h'}$ and hence $L_{0}$ is affine on $[0,h']$. Since $L_{0}$ is convex and even on $[-h',h']$ it has to be non decreasing on $[0,h']$.
\end{proof}
Notice that there exist $r_{A}\in\Z$ and  a  holomorphic $\th_{A}:\Omega^+_{h'}\to \C$ such that for any $z\in \Omega^+_{h'}\to \C$
$$\l_{A}(z) =\exp(2\pi ir_{A}z+\theta_{A}(z)).$$
By Cauchy formula $\int_{\T}\theta_{A}(x+i\nu)dx$ is independent of $\nu$; let us call $\bar \theta_{A}$ this value. We have
$$\int_{\T}\log |\l_{A}(x+i\nu)|dx=2\pi r_{A}\nu+\Re\bar \th_{A}$$
and since this quantity is non negative and non decreasing on $(0,h']$ we have $r_{A}\geq 0$ and $\Re\bar\th_{A}\geq 0$. When $A$ is of mixed type $\Re\bar\th_{A}>0$.

When $A$ is regular we can diagonalize $A$ in a holomorphic way:
\begin{lem}If $A\in C^\w_h(\T,SL(2,\R)))$ is regular on $\Omega_{h'}^+$, then for any $0<\nu<h'$,   there exists $B_{\nu}\in C^\infty(\T,GL(2,\C))$ such that $B_{\nu}(\cdot)^{-1}A(\cdot+i\nu)B_{\nu}(\cdot)=D_{\l_{A}(\cdot+i\nu )}$ where  we use the notation $D_{\l}={\rm diag}(\l,\l^{-1})$.
\end{lem} 
\begin{proof} 
Let $C_{\nu}$ be the circle $(\R+i\nu)/\Z$.
For any $z\in C_{\nu}$  the eigenvalues $\l_{A}(z)$, $\l_{A}(z)^{-1}$ of $A(z)$, are distinct and we can thus define $E^\pm(z)=\ker (A(z)-\l_{A}^\pm(z)I)\subset\C^2$ the corresponding one-dimensional complex eigenspaces.  The maps $C_{\nu}\ni z\mapsto E^\pm(z)\in \mathbb{P}^1(\C)$ are smooth   and   define two smooth complex line bundles $E^+\to C_{\nu}$ and  $E^-\to C_{\nu}$. The existence of a smooth function $B_{\nu}:C_{\nu}\to GL(2,\C)$ such that   on $C_{\nu}$ one has $B_{\nu}(\cdot)^{-1}A(\cdot)B_{\nu}(\cdot)=D_{\l_{A}(\cdot)}$ is equivalent to the fact that there exist two nonzero smooth sections $C_{\nu}\ni z\mapsto u^\pm(z)\in E^\pm(z)$. To conclude, we could  just argue that any complex fibered bundle over the circle is equivalent to a trivial one (see for example \cite{Ha}) but here we can give a direct argument. Since the manifold $C_{\nu}$ has real  dimension 1, $\mathbb{P}^1(\C)$ has real dimension 2 and $E^{\pm}:C_{\nu}\to \mathbb{P}^1(\C)$ are smooth, there exists a complex line $\mathbb{P}^1(\C)\ni \beta=[\beta_{1},\beta_{2}]\notin E^\pm(C_{\nu})$. For  each point $t\in C_{\nu}$ the affine lines of $\C^2$, $E^{\pm}(t)$ and  $\{(u_{1},u_{2})\in\C^2:\beta_{1}u_{1}+\beta_{2}u_{2}=1\}$ intersect in a single point which depends smoothly on $t$. This defines the searched for smooth sections $u^\pm(\cdot)$. 

\end{proof}

We use the same notations as before. 

\begin{lem}\label{prop:mixed}Let $\ep>0$ and $0<\nu<h'<h$, $p/q\in\Q$, $A\in C^\omega_{h}(\T,SL(2,\R))$ and assume that  $(p/q,A)^q$ is   regular.  Then,  there exists $\eta_{0}:=\eta_{0}(A,q,\ep,\nu)>0$, continuous with respect to $A\in C^\omega_{h}(\T,SL(2,\R))$,  such that for any irrational $\a$ satisfying $0<|\a-(p/q)|<\eta_{0}$ one has
$$|LE((\a,A(\cdot+i\nu))^q)-LE((p/q,A(\cdot+i\nu))^q)|\leq \ep. $$
\end{lem}
\begin{proof}Define $A^{(q)}$ by $(p/q,A)^q=(0,A^{(q)})$; by assumption this is a regular map.
The previous Lemma shows that there exists $B_{\nu}\in C^\infty(\T,GL(2,\C))$ such that  $(0,B_{\nu})^{-1}\circ (p/q,A_{\nu})^q\circ (0,B_{\nu})=(0, D_{\l_{A^{(q)}}(\cdot+i\nu)}(\cdot))$, thus
$$(0,B_{\nu})^{-1}\circ (\a,A_{\nu})^q\circ (0,B_{\nu})=(q\a, D_{\l_{A^{(q)}}(\cdot+i\nu)}(\cdot)P(\cdot))$$ where the perturbation $P:\T\to SL(2,\R)$ satisfies 
$$\|P-I\|_{0}\lesssim  q|\a-p/q|(1+\|A\|_{h})^qh^{-2}\|B_{\nu}\|_{0}\|B_{\nu}^{-1}\|_{0}.$$ Since the matrices $D_{\l_{A^{(q)}}(\cdot+i\nu)}$ are uniformly hyperbolic ($\min_{\T}|\lambda_{A^{(q)}}(\cdot+i\nu)|>0$) it is not difficult to show that the cocycle $(\a,A_{\nu})^q$ is uniformly hyperbolic if $q|\a-p/q|$ is small enough and that its Lyapunov exponent is $\ep$-close to 
$LE((p/q,A(\cdot+i\nu))^q)$. 
\end{proof}

\begin{prop}Let $q\in\N^*$,  $A\in C^\omega_{h}(\T,SL(2,\R))$ be such that $(p/q,A)^q$ is    regular and let $\ep>0$.  Then,  there exists $\eta_{1}:=\eta_{1}(A,h,q, \ep)>0$, continuous with respect to $A\in C^\omega_{h}(\T,SL(2,\R))$,   such that for any $0<|\a-(p/q)|<\eta_{1}$  one has 
$$qLE(\a,A)\geq LE((p/q,A)^q)-\ep.$$
\end{prop}
\begin{proof}Choose $0<\nu_{1}<\nu_{2}<h'$ with $\nu_{1}=h'/4$, $\nu_{2}=h'/2$ and introduce
$$\ti L_{q\a}(\nu)=LE((\a,A(\cdot+i\nu))^q)=qLE(\a,A(\cdot+i\nu))$$
$$ \ti L_{0}(\nu)=LE((p/q,A(\cdot+i\nu))^q)=LE(0,A^{(q)}(\cdot+i\nu)).$$
From Lemma \ref{lem:cvx} we know that $\nu\mapsto \ti L_{q\a}(\nu)$ is a convex function and by Lemma \ref{lem:reglem} that $\nu\mapsto \ti L_{0}(\nu)$ is an affine function.
We now define  $\e_{1}=\min(\e_{0}(A,\nu_{1},\ep/3),\e_{0}(A,\nu_{2},\ep/3))$ and we apply Lemma \ref{prop:mixed}:  for any irrational $0<|\a-(p/q)|<\eta_{1}$, $|\ti L_{q\a}(\nu_{j})- \ti L_{0}(\nu_{j})|\leq \ep$, $j=1,2$. 
The function $\nu\mapsto \ti L_{q\a}(\nu)$ is convex  on $\nu\geq 0$, hence for $t<0$, 
\begin{align*}  \ti L_{q\a}(\nu_{1}+t(\nu_{2}-\nu_{1}))&\geq  \ti L_{q\a}(\nu_{1})+t(\ti L_{q\a}(\nu_{2})-\ti L_{\a}(\nu_{1}))\\
&\geq   \ti L_{0}(\nu_{1})+t(\ti L_{0}(\nu_{2})-\ti L_{0}(\nu_{1}))+(2t-1)\ep/3
\end{align*}
and since $\nu\mapsto \ti L_{0}(\nu)$ is affine
one has for $t=-\nu_{1}/(\nu_{2}-\nu_{1})=-1$, 
$$\ti L_{q\a}(0)\geq \ti L_{0}(0)-\ep.$$
 \end{proof}

\subsection{Proof of  Theorem \ref{theo:5.1}}
We can now complete the proof of Theorem \ref{theo:5.1}. Indeed, if $A^{(q)}$ is regular  and of mixed type $LE(0,A^{(q)})>0$ and we can apply the previous Proposition with $\ep=LE(0,A^{(q)})/2$.
\section{Coexistence of ac components and pp components}
Let  $(p,q)\in\Z\times \N^*$, $p\wedge q=1$, and $V:\T \to \R$ a nonnegative smooth function the support of which is included in $]0,1/q[$. We shall assume in Sections \ref{sec:5.3} and \ref{sec:5.4} that in addition $V\in \cP^\infty$. 
\subsection{Computation of $(p/q,S_{E-V})^q$}

We introduce the matrices  $$C_{E}=\bm E& -1\\ 1& 0\em,\quad S_{E-V}(x)=\bm E-V(x)& -1\\ 1& 0\em$$ and 
$$A_{E}^{(q)}(x)=S_{E-V}(x+(q-1)\frac{p}{q})\cdots S_{E-V}(x).$$
Denote by $I_{j}$ the arc $[j/q,(j+1)/q]$, $0\leq j\leq q-1$ and  let $\bar j\in \{0,\ldots,j-1\}$ be such that $\bar j p\equiv j\mod q$ (if $\bar p$ is the inverse of $p$ in $\Z/q\Z$, $\bar j= j\bar p$). 
Notice that  if $x\in I_{j}$, 
\be A_{E}^{(q)}(x)=C_{E}^{q-1-\bar j}S_{E-V}(\ti x_{j})C_{E}^{\bar j} \label{1}\ee 
where $\ti x$ is the point of the orbit of $x$ under $x\mapsto x+(p/q)$ that lies in $I_{0}$ and as a consequence the trace $\tau_{E}(x)$  of $A_{E}^{(q)}(x)$ is equal to  the trace of $\bm E-V(\ti x) & -1\\ 1& 0\em\bm E & -1\\ 1& 0\em^{q-1}$.
\begin{lem}\label{lem:6.1}If $E=2\cos\th$ (resp. $E=2\cosh\th$) one has 
$$\tr(A_{E}^{(q)}(x))=-V(\ti x) \frac{\sin(q\th)}{\sin\th}+2\cos (q\th).$$
(resp. 
$$\tr(A_{E}^{(q)}(x))=-V(\ti x) \frac{\sinh(q\th)}{\sinh\th}+2\cosh (q\th).)$$

\end{lem}
\begin{proof}
We assume that $E=2\cos \th$, so the matrix $C_{E}:=\bm E & -1\\ 1& 0\em$ is elliptic with eigenvalues $e^{\pm i \th}$. We find that 
$\bm m \\ 1\em$ and $\bm \bar m \\ 1\em$ are the eigenvectors corresponding to $e^{i\th}$ and $e^{-i\th}$ where $m=e^{i\th}$. Hence 
$$C_{E}=\frac{1}{m-\bar m}\bm m& \bar m\\ 1&1\em\bm e^{i\th} & 0\\ 0& e^{-i\th}\em \bm 1 &-\bar m\\ -1& m\em$$
and if we  denote $\ph=(q-1)\th$
$$C_{E}^{q-1}=\frac{1}{2i\sin\th }\bm e^{i\th}& e^{-i\th}\\ 1&1\em\bm e^{i\ph} & 0\\ 0& e^{-i\ph}\em\bm 1& -e^{-i\th}\\ -1& e^{i\th}\em.$$ 
We find
$$C_{E}^{q-1}=\frac{1}{\sin\th}\bm\sin(\th+\ph)& -\sin(\ph)\\ \sin\ph & \sin(\th-\ph)\em.$$
So
\begin{multline*}\bm E-V(\ti x) & -1\\ 1& 0\em\bm E & -1\\ 1& 0\em^{q-1}=\\\frac{1}{\sin\th}\bm(E-V)\sin(\th+\ph)-\sin\ph& -(E-V)\sin\ph-\sin(\th-\ph) \\ \sin(\th+\ph) & -\sin\ph\em.
\end{multline*}
Using that $\ph=(q-1)\th$ and $E=2\cos\th$ this last matrix can be written
$$\bm\cos(q\th)+(\cos\th-V)\displaystyle\frac{\sin(q\th)}{\sin\th}& V\biggl(\displaystyle\frac{\sin(q\th)}{\sin\th}\cos\th-\cos(q\th)\biggr) \\ \displaystyle\frac{\sin(q\th)}{\sin\th} & \cos(q\th)-\displaystyle\frac{\sin(q\th)}{\sin(\th)}\cos\th\em,
$$
and thus
$$\tau_{E}(x)=-V(\ti x) \frac{\sin(q\th)}{\sin\th}+2\cos (q\th).$$
Notice that the same computations with $E=2\cosh\th$ will give
$$\tau_{E}(x)=-V(\ti x) \frac{\sinh(q\th)}{\sinh\th}+2\cosh (q\th).$$

\end{proof}

\subsection{Existence of ac spectrum}
Let $\th_{k}=\pi k/q$, $0\leq k\leq 2q-1$. Close to each $\th_{k}$, $1\leq k\leq 2q-1$,  there exist  closed intervals with nonempty interior $J_{k}$, such that for every $\th\in J_{k}$
\begin{itemize}
\item The quantities $\sin(q\th)/\sin\th$ and $\cos(q\th)$ have the same sign;
\item $(\max_{\T}V)|\sin(q\th)/\sin(\th)|<1/2$;
\item $2>|2\cos(q\th)|>3/2$.
\end{itemize}
Let $\delta_{k}=\min_{\th\in J_{k}}2-2|\cos(q\th)|$.
\begin{lem}For any $E\in 2\cos J_{k}$, the cocycle $(p/q,S_{E-V})^q$ is totally elliptic and $E\mapsto \rho(p/q,S_{E-V})$ is a real analytic  non constant function on  $2\cos J_{k}$.
\end{lem}
\begin{proof}
For the first part of the statement we observe that 
 since the function $V$ is nonnegative we have by Lemma \ref{lem:6.1} that  for any $\th\in J_{k}$ and any $x\in\T$
$$|\tr(A_{E}^{(q)}(x))|\leq 2-\delta_{k}.
$$
For the second part we observe that
$$\rho(p/q,S_{E-V})=(1/q)\rho(0,A_{E}^{(q)})$$
and notice that if $\tau_{E}(x)=\tr(A_{E}^{(q)}(x))$
$$\rho(0,A_{E}^{(q)})=\frac{1}{2\pi}\int_{\T}\arccos( \tau_{E}(x)/2)dx$$
which is certainly a non-constant real analytic function of $E\in 2\cos(J_{k})$.
\end{proof}

We now get as a corollary of Theorem \ref{theo:mainac}:
\begin{cor}There exist $\eta_{ac}>0$ and in each interval $2\cos J_{k}$ a set  $\cA_{k}$ of positive Lebesgue measure  such that for all $\a \in D_{p/q}(\eta_{ac})$ and all $E\in \bigcup_{k}\cA_{k}$, the cocycle $(\a, S_{E-V})$ is $C^\infty$ conjugated to a constant  elliptic cocycle with rotation number in $DS_{\a}\setminus\{0\}$. The same conclusion is true for any smooth perturbation $\ti V$ of $V$ which is  sufficiently $C^{s_{0}}$-close to  $V$.
\end{cor}

We  deduce from the last corollary and the criterium for ac spectrum  Theorem \ref{critac} the following corollary.
\begin{cor}\label{statementac}Let $\ti V\in C^\omega(\T,\R)$ be  such that $\|\ti V-V\|_{C^{s_{0}}}$ is small enough. Then, for any $\a\in D_{p/q}(\eta_{ac})$ there exist nontrivial intervals $I_{ac,k}\subset 2\cos J_{k}$ such that  for any $x\in\T$ the restriction of the spectral measure  $\mu_{\ti V,\a,x}$ on  each interval $I_{ac,k}$ is ac.
\end{cor}
\subsection{Existence of pp spectrum}\label{sec:5.3}

Let $\ti V\in C^\omega(\T,\R)$ be such that $\|\ti V-V\|_{C^{s_{0}}}$ is small enough; there exists $h>0$ such that $\ti V\in C^\omega_{h}(\T,\R)$.  We define $\ti A_{E}^{(q)}$ by $(p/q,S_{E-\ti V})^q=(0,\ti A_{E}^{(q)})$ and we recall that we denote by $K:=K(V)=\max_{\T} V$.
\begin{lem} For $E=2\cosh\th\in [10,K-10]$ the map $\ti A^{(q)}_{E}$ is regular and of mixed type.
\end{lem}
\begin{proof}
For $E=2\cosh\th>2$ we have 
\begin{align*}\tr(A_{E}^{(q)}(x))&=-V(\ti x) \frac{\sinh(q\th)}{\sinh\th}+2\cosh (q\th)\\
&=-V(\ti x) e^{(q-1)\th}(1+a(\th))+e^{q\th}(1+b(\th))\\
&=e^{(q-1)\th}(-V(\ti x)(1+a(\th))+e^\th(1+b(\th)))
\end{align*}
where $|a(\th)|\leq e^{-2\th}$ and $|b(\th)|\leq e^{-2q\th}$.
We observe that $|\tr(A_{E}^{(q)}(x))|\leq 2$ if and only if
$$\frac{1}{1+a(\th)}(e^\th(1+b(\th))-2e^{-(q-1)\th})\leq V(\ti x)\leq \frac{1}{1+a(\th)}(e^\th(1+b(\th))+2e^{-(q-1)\th})
$$
The assumptions on $V$ show that the last inequalities can only occur  at points  $\ti x$ where  $V'(\ti x)\ne 0$; at these points   $(d/dx)\tr(A_{E}^{(q)}(x))\ne 0$. Now if  $\|\ti V-V\|_{C^1}$ is small enough the same conclusion will hold for $\ti A_{E}^{(q)}$: if $|\tr(\ti A_{E}^{(q)}(x)|\leq 2$ then $(d/dx)\tr(\ti A_{E}^{(q)}(x))\ne 0$. By Lemma \ref{lem:reg} we conclude that $\ti A_{E}^{(q)}$ is regular. It is also of mixed type as is clear from the previous inequalities.

\end{proof}

\begin{cor}\label{cor:5.6} Assume that $\|\ti V-V\|_{C^1}$ is small enough. There exists $\eta_{pp}(K,q)>0$ such that for any $E\in [10,K-10]$ and   any $\a\in\R\setminus\Q$ such that  $|\a-(p/q)|<\eta_{pp}$ 
$$LE(\a,S_{E-\ti V})>0.
$$
\end{cor}
\begin{proof} This follows from  the previous Lemma and Theorem \ref{theo:5.1}.
\end{proof}

\begin{cor} \label{statementpp}Given any $\ti V\in C^\omega(\T,\R)$ such that  $\|\ti V-V\|_{C^1}$ is small enough there exists a set $\mathcal{BG}(\ti V,p/q)$ of full Lebesgue measure in $](p/q)-\eta_{pp},(p/q)+\eta_{pp}[$ such that for any $\a\in\mathcal{BG}$ the operator $H_{\ti V,\a,0}$ has p.p. spectrum in $[10,K-10]$.
\end{cor}
\begin{proof}This follows from the previous Corollary \ref{cor:5.6} and Theorem \ref{theo:BG}. Notice that by Lemma \ref{lem:A2} the intersection of the spectrum of  $H_{\ti V,\a,0}$ with $[10,K-10]$ is nonempty. 
\end{proof}

\subsection{Coexistence of ac and pp spectrum - Proof of Theorem \ref{MainA}}\label{sec:5.4}
We can now conclude the proof of Theorem \ref{MainA} by combining Corollaries \ref{statementac} and \ref{statementpp}.
Indeed, for $\eta:=\min(\eta_{ac},\eta_{pp})$ the set $D_{p/q}(\eta)\cap \mathcal{BG}(\ti V,p/q)$ has positive Lebesgue measure and Lebesgue density one at $p/q$.

\section{Proof of Theorem \ref{MainC}}

\subsection{The a.c. part of the spectrum}
We apply the results of the preceding subsections to cocycles of the form $(\a,S_{E-V_{K,\lambda}})$. The following two lemmas assure  that the cocycle $(1/2,S_{E-V_{K,\l}}(\cdot))$ satisfies the assumption of Theorem \ref{theo:mainac}.
\begin{lem}For any $K>0$ we have that
$$
\rho(\a, S_{E-V_{K,\l}})\to\frac{1}{2\pi}\arccos(E/2) \text{ as } \lambda\to \infty,
$$
uniformly in $\alpha$ and $E$.
\end{lem}
\begin{proof}We know that $\frac{1}{2\pi}\arccos(E/2)$ is the rotation number of the free problem, i.e., 
the case when there is no potential. Since, by taking $\lambda$ large, 
$V_{K,\lambda}(x)$ can be made arbitrarily small outside any open interval containing $x=0$, it follows from the continuity of the rotation number (w.r.t. the fibered map), see Lemma \ref{lem:var1}, that $\rho(\a, S_{E-V_{K,\l}})$ must converge to the rotation number
of the free problem as $\lambda\to\infty$. 
\end{proof}

\begin{lem}Given $K>10$ we have that for all $\lambda$ large
$$
\text{tr} (S_E(x+\alpha)S_E(x))\in [-2+1/K^2,2-1/K^2] ~\forall x\in \mathbb{T}
$$
for all $E\in [-3/(2K),-1/K]$ and  $\alpha\in [1/4,3/4]$. In particular,  for any $x\in\T$
the matrix $S_E(x+\alpha)S_E(x)$ is elliptic.
\end{lem}
\begin{proof}We have
$$
\text{tr} (S(x+\alpha)S(x))=-2+(E-V_{K,\l}(x+\alpha))(E-V_{K,\l}(x)).
$$
\end{proof}
\begin{cor} There exists  $\l_{ac}$ such that for any $\l>\l_{ac}$, there  exists $\e_{ac}>0$ such that for any $0<\e\leq \e_{ac}$ and  any  $\a\in D_{1/2}(\e)$ the operator $H_{V_{K,\l},\a,x}$ has some a.c. spectrum in  $[-3/(2K),-1/K]$ for all $x\in\T$.
\end{cor}
\begin{proof}
 This follows from Theorems \ref{theo:mainac} and \ref{critac}.

\end{proof}
\bigskip

\subsection{The p.p. part of the spectrum}
\bigskip
Next we prove the following: 

\begin{prop}
For every $\varepsilon>0$ there is a 
$\lambda_{pp}=\lambda_{pp}(\varepsilon)>0$ such that for all $\lambda>\lambda_{pp}$ we have for all irrational $\a$
$$
LE(\a,S_{E-V_{K,\lambda}})>0\quad \text{for all } |E|>2+\varepsilon.
$$
\end{prop}
\begin{rem}
Note that the constant $K$ does not appear in this estimate, nor the $\alpha$.
\end{rem}

\begin{proof}
We let
$$
f(z)=\frac{K}{1+\lambda(2-z-1/z)}.
$$
Then $f(e^{2\pi i x})=V(x)$. The function $f$ has exactly two poles:
$$
z_0=\frac{2\lambda+1-\sqrt{4\lambda+1}}{2\lambda},
\quad z_1 = \frac{2\lambda+1+\sqrt{4\lambda+1}}{2\lambda}.
$$
Note that $0<z_0<1<z_1$ and $z_0,z_1\approx 1$ if $\lambda\gg 1$. 
With this notation we can write
$$
f(z)=-\frac{Kz}{\lambda(z-z_0)(z-z_1)}=\frac{Cz}{(z-z_0)(z-z_1)},
$$
where $C=-K/\lambda$. We let
$$
B_E(z)=\left( \begin{matrix}
E-f(z) & -1 \\ 1 & 0
\end{matrix} 
\right),  M_E(z)=\left( \begin{matrix}
E(z-z_0)-\frac{Cz}{z-z_1} & -(z-z_0) \\ (z-z_0) & 0
\end{matrix} 
\right).
$$
Then 
$$
B_E(z)=\frac{1}{z-z_0}M_E(z).
$$
Note that $M_E(z)$ is analytic in the disc $|z|<z_1$, and recall that $z_1>1$.

Let $e_j=e^{2\pi i j\alpha}$ ($j=1,2,\ldots$), and let
$$
\begin{aligned}
B_E^n(z)&=B_E(e_{n-1}z)\cdots B_E(z)\\
M_E^n(z)&=M_E(e_{n-1}z)\cdots M_E(z).
\end{aligned}
$$
Then $B_E^n(e^{2\pi i x})=S_E^n(x)$. 

We now use Herman's trick. We have
$$
\begin{aligned}
\int_0^1\log\|S_E^n(x)\|dx=\int_0^1\log\|B_E^n(e^{2\pi i x})\|dx=\\
\sum_{j=0}^{n-1}\int_0^1\log|e^{2\pi i (x+j\alpha)}-z_0|dx
+\int_{0}^1\log\|M_E^n(e^{2\pi ix})\|dx\geq \log\|M_E^n(0)\|,
\end{aligned}
$$
since each term in the sum is equal to zero, by Jensen's formula, and since $\log\|M_E^n(z)\|$ is
subharmonic in $|z|<z_1$. We note that
$$
M_E(0)=z_0\left( \begin{matrix}
-E  & 1 \\ -1 & 0
\end{matrix} 
\right)
$$
and
$$
M_E^n(0)=z_0^n\left( \begin{matrix}
-E  & 1 \\ -1 & 0
\end{matrix} 
\right)^n.
$$
For $|E|>2$ the matrix $M_E(0)$ has an eigenvalue $\mu(E)$ such that 
$$|\mu(E)|=z_0\left(\frac{|E|+\sqrt{E^2-4}}{2}\right).$$ Thus, by the spectral radius formula,
$$
\lim_{n\to\infty}\frac{1}{n}\log\|M_E^n(0)\|=\log|\mu(E)|.
$$
We have hence shown that
$$
LE(\a,S_{E-V_{K,\lambda}})\geq \log|\mu(E)|\quad  \text{ for all } |E|>2.
$$

\end{proof}

We thus get from Bourgain-Goldstein (\cite{BG}) and Lemma \ref{lem:A2} 

\begin{cor} For $\l>\l_{pp}$ and a.e. $\alpha\in \mathbb{T}$, the spectrum of the operator
$H_{V_{K,\l},\a,0}$ in the interval $[3,\infty)$ is pure point, and the eigenfunctions decay exponentially fast. 
\end{cor}

\subsection{Proof of Theorem \ref{MainC} } Just take $\l>\max(\l_{ac},\l_{pp})$.
\hfill $\Box$

\appendix
\section{Locating the spectrum}

We recall that if $\a$ is irrational $\spec(H_{V,\a,x})$ is independent of $x$. We call $\Sigma_{V,\a}$ this non-empty, compact subset of $\R$.
\begin{lem}Let $\a$ be irrational and $V:\T\to\R$ be continuous. Then for any $E\in [\min_{\T} V,\max_{\T}V]$ one has 
$$ [E-\sqrt{2},E+\sqrt{2}]\cap \Sigma_{V,\a}\ne \emptyset.
$$
\end{lem}
\begin{proof} Let $x$ be such that $V(x)=E$. If $E\in \spec(H_{V,\a,x})$ there is nothing to prove. Otherwise, $H_{V,\a,x}-E$ is invertible with bounded inverse and by the Spectral Theorem  ($H_{V,\a,x}$ is bounded symmetric)
\be\|(H_{V,\a,x}-E)^{-1}\|\leq \frac{1}{\dist(E,\Sigma_{V,\a})}.\label{thspec}
\ee
We now observe that if  $\delta_{k}$ is the delta function at $k\in\Z$ ($\delta_{k}(n)=1$ if $n=k$ and $\delta_{k}(n)=0$ otherwise) one has $v:=(H_{V,\a,x}-E)\delta_{0}=\d_{1}+\d_{-1}+(V(x)-E)\d_{0}=\d_{1}+\d_{-1}$. By (\ref{thspec}) one has
$$ \|\d_{0}\|=\|(H_{V,\a,x}-E)^{-1}v\|\leq  \frac{1}{\dist(E,\Sigma_{V,\a})}\|v\|
$$ and since $\|\delta_{0}\|=1$, $\|v\|=\sqrt{2}$ one gets 
$$\dist(E,\Sigma_{V,\a})\leq \sqrt{2}.
$$
\end{proof}
As a corollary we get 
\begin{lem}\label{lem:A2}Let $V:\T\to\R$ be a continuous function such that $V(\T)=[0,K]$ with $K\geq 5$. For all $\a\in\R\setminus\Q$, $\Sigma_{V,\a}\cap [3,K-2]\ne\emptyset$.
\end{lem}

\section{Conjugating elliptic matrices to rotations}
\begin{lem}\label{app:1}
If $\ph,\ti\ph:\T^d\to \R$ are smooth functions satisfying $$\left(\cos (\ti\ph(x))-\cos(\ph(x))\right)^2+(\sin(\ti\ph(x))+\sin(\ph(x)))^2\ne 0$$ (or equivalently $\forall x\in\T^d$, $\ph(x)+\ti\ph(x)\notin 2\pi\Z$) and if $C:\R\to SL(2,\R)$ satisfies 
$$R_{\ti\ph(x)}=C(x)R_{\ph(x)}C(x)^{-1}$$
then there exists a smooth function $\psi:\R\to\R$ such that $C(\cdot)=R_{\psi(\cdot)}$. Furthermore, if $C$ is $\Z$-periodic and homotopic to the identity then $\psi$ is also $\Z$-periodic.
\end{lem}
\begin{proof}The first part  is a simple computation and the second is just due to the fact that if $R_{\psi}$ is periodic and homotopic to the identity then $\psi$ has to be $\Z$-periodic.
\end{proof}
\begin{lem}\label{app:0}
Let $A:\T\to SL(2,\R)$ a smooth map such that $\max_{x\in\T}|{\rm tr}(A(x)|<2$. Then, there exist smooth maps $B:\T\to SL(2,\R)$ and $a:\T\to ]0,\pi[$ such that for all $x\in\T$
$$B(x)^{-1}A(x)B(x)=R_{a(x)}.
$$
\end{lem}
\begin{proof}Let $\cE=\{A\in SL(2,\R),\ |{\rm tr} A|<2\}$ be the set of elliptic matrices of $SL(2,\R)$.  The map $S:SL(2,\R)\times ]0,\pi[\to \cE$, $(B,a)\mapsto BR_{a}B^{-1}$ is a submersion. There thus exist smooth maps $a:\R\to ]0,\pi[$, $\widetilde B:\R\to SL(2,\R)$ such that 
\be \forall \ x\in\R,\ A(x)=\widetilde B(x)R_{a(x)}{\widetilde B}(x)^{-1}.\label{11}
\ee
 Since $2\cos a(x)={\rm tr}A(x)$ we have $a(x+1)=a(x)$. Now, since  $A$ is 1-periodic, identity (\ref{11}) shows that  for any $x\in\R$, $\widetilde B(x)^{-1}\widetilde B(x+1)$ commutes with $R_{a(x)}$ and  by Lemma \ref{app:1} there exists a smooth $b:\R\to\R$ such that  $\widetilde B(x+1)=\widetilde B(x)R_{b(x)}$.  It is not difficult to show that there exists a smooth map $c:\R\to\R$ such that $-b(x)=c(x+1)-c(x)$.  Setting $B(x)=\widetilde B(x)R_{c(x)}$ we have 
 \be A(x)=B(x)R_{a(x)}B(x)^{-1}
 \ee
 with $B(x+1)=\widetilde B(x+1)R_{c(x+1)}=\widetilde B(x)R_{b(x)+c(x+1)}=\widetilde B(x)R_{c(x)}=B(x)$.

\end{proof}
In fact one can prove the following more quantitative lemma:
\begin{lem} \label{app:2}  If $\ph:\T^d\to\R$ smooth  satisfies $\forall x\in\T^d,\ R_{\ph(x)}\ne \pm I$, then there exists $\e>0$ (depending on $\d:=\dist_{\R/\pi\Z}(\ph(\T^d),0)$) such that for any $F:\T^d\to sl(2,\R)$ satisfying $\|F\|_{0}\leq \e$ one can find  smooth maps $\psi:\T^d\to (0,\pi)$, $Y:\T^d\to sl(2,\R)$ such that $R_{\ph(\cdot)}e^{F(\cdot)}=e^{Y(\cdot)}R_{\psi(\cdot)}e^{-Y(\cdot)}$ and 
$$\|\psi\|_{k} \leq C_{k}(\|\ph\|_{k},\e,\d) \|F\|_{k},\qquad \|Y \|_{k} \leq C_{k}(\|\ph\|_{k},\e,\d)\|F\|_{k} $$ 

\end{lem}

\section{Variations of the fibered rotation number}

\subsection{Dependence w.r.t. the frequency}\label{sec:depfreq} We refer to \cite{He} for more details on the materials of this section.
We denote by $ Homeo^+(\bS^1)$ the group of orientation preserving homeomorphisms of the circle and by $D(\bS^1)$ the set of  orientation preserving homeomorphisms  $h$ of $\R$ such that $h(\cdot+1)=1+h(\cdot)$. Any $f\in Homeo^+(\bS^1)$ has a lift $\ti f$ in $D(\bS^1)$ (the projection is $\R\to\R/\Z$, $x\mapsto x+\Z$).

Let $(X,T)$ be a uniquely ergodic topological  dynamical system and $f:X\times \bS^1 \to X\times \bS^1$, $(x,y)\mapsto (Tx, f(x,y))$  a continuous homeomorphism,  such that for any $x\in X$, $y\mapsto f(x,y)$ is an orientation preserving homeomorphism of $\bS^1$ and assume that  the map $X\to Homeo^+(\bS^1)$, $x\mapsto f(x,\cdot)$ is homotopic to the identity (the groups  $Homeo^+(\bS^1)$, $D(\bS^1)$ are  endowed with the topology of uniform convergence; these topologies are metrizable).  We denote by $f_{x}:\bS^1\to\bS^1$ the homeomorphism $y\mapsto f(x,y)$.  Under these assumptions, $f$ admits a lift  $\ti f:X \times \R\toitself$, $\ti f:(x,y)\mapsto (Tx, y+\ph(x,y))$ where $\ph:X\times\R/\Z\to\R$ is continuous; we shall use the notation $\ti f=(T,\ph)$. The fibered rotation number $\rho(\ti f)$ is the limit $\lim_{n\to\infty}(1/n)({\ti f}^n(x,y)-y)=\lim_{n\to\infty}(1/n)\sum_{k=0}^{n-1}\ph({\ti f}^k(x,y))$ which is independent of $(x,y)$. A useful fact is that this fibered rotation number coincide with $\int_{X\times \bS^1}\ph d\nu$ for any $f$-invariant probability measure $\nu$ on $X\times \bS^1$. We denote $\rho(f)=\rho(\ti f)\mod 1$. When $X$ is reduced to a point we recover the notion of the rotation number of an orientation preserving homeomorphism of the circle.

In the  case  $X=\T$ and $T=T_{\a}:\T\to\T$ is the translation $x\mapsto x+\a$, $T_{\a}$ is uniquely ergodic if and only if $\a$ is irrational and in that case the unique invariant measure by $T_{\a}$ is the Haar measure on $\T$. When $\a=p/q$ is rational we extend the definition of the rotation number by setting $\rho(\ti f)=(1/q)\int_{\T}{\rm rot }({\ti f^{(q)}}_{x})dx$ where ${\rm rot }(\ti f^{q}_{x})$ is the rotation number of  the  $q$-iterate of the homeomorphism  $\ti f_{x}\in D(\bS^1)$ defined by  ${\ti f}^q:(x,y)\mapsto (x,{\ti f}^q_{x}(y))$. 

If $\a\in\T$ and $\ph:\T\times\R/\Z\to \R$ we denote by $(\a,\ph):\T\times \R\toitself$, $(x,y)\mapsto (x+\a,y+\ph(x,y))$. 

\begin{lem}The map  $\Phi:\T\times C^0(\T,D(\bS^1))\to \R$, $(\a,h)\mapsto \rho(\a,h)$ is continuous.
\end{lem}
\begin{proof}Herman proved the continuity of this map restricted to $(\R-\Q)/\Z\times C^0(\T,D(\bS^1))$  (\cite{He}, Prop. 5.7.)  so it is enough to prove that for any rational number $\a:=p/q$, any $h\in C^0(\T,D(\bS^1))$, any sequence of irrational numbers $\a_{n}\to\a$  and any sequence $h_{n}\in C^0(\T,D(\bS^1))$ converging uniformly to $h$, one has $\Phi(\a_{n},h_{n})\to \Phi(\a,h)$.  Since for $\beta$ irrational $\rho((\beta,\ph)^q)=q\rho(\beta,\ph)$ it is enough to treat the case $\a=0$, what we shall do. Also, it is enough to establish that for any such sequence, there is a subsequence for which the convergence holds. Let $\nu_{n}$ be a sequence of probability measure  invariant by $(x,y)\mapsto (x+\a_{n},y+\ph_{n}(x,y))$. By unique ergodicity of $x\mapsto x+\a_{n}$ we have $\pi_{*}\nu_{n}=m$ where $\pi:(x,y)\mapsto x$ and $m$ is the Haar measure on $\T$. For a subsequence $n_{k}$, the sequence of probability measures $\nu_{n_{k}}$ converges weakly  to a probability measure $\nu$ which is invariant by $(x,y)\mapsto (x,y+\ph(x,y))$. From $\pi_{*}\nu_{n_{k}}=m$ we get $\pi_{*}\nu=m$. Thus, the measure $\nu$ can be decomposed as 
$\nu=\int_{\T}dx\otimes \nu_{x}$
 and  for Lebesgue a.e. $x\in\T$, the probability measures $\nu_{x}$ are invariant by $\ti f_{x}: y\mapsto y+\ph(x,y)$; by the theory of circle homeomorphisms ({\it cf.} \cite{He-ihes})  $\int\ph(x,y)d \nu_{x}(y)={\rm rot}{\ti f}_{x}$, hence $\int \ph d\nu=\rho(0,\ph)$. Since $\nu_{n_{k}}$ converges weakly to $\nu$ and $\ph_{n_{k}}$ converges uniformly to $\ph$ we have $\rho(\a_{n_{k}}, \ph_{n_{k}})=\int \ph_{n_{k}}d\nu_{n_{k}}\to\int\ph d\nu=\rho(0,\ph)$, which is what we wanted to prove.
\end{proof}

\subsection{Lipschitz dependence of the  fibered rotation number on parameters}
With the notations of section \ref{sec:depfreq} let us now assume that our homeomorphism  $f:X\times\bS^1\toitself$ depends on a parameter $E\in I$ where $I$ is an interval of $\R$; more specifically,  we assume that $X\times\bS^1\times I\to X\times\bS^1$, $(x,y,E)\to (Tx,f_{E}(x,y))$ is  continuous, that for any $E\in I$ and any $x\in X$, $f_{E}(x,\cdot)\in Homeo^+(\bS^1)$, and that for any $E\in I$, $X\to D(\bS)$, $x\mapsto f_{E}(x,\cdot)$ is homotopic to the identity. 
We also assume that $f$ is Lipschitz w.r.t to the variable $E$, the Lipschitz constant being uniform with respect to the variables $x$ and $y$.
\begin{lem}If for some $E_{0}\in I$  there exists $b:X\times S^1\toitself $  continuous (w.r.t. $(x,y)$),  such that for every $x\in X$,  $b_{x}^{-1}$ is Lipschitz with respect to the variable $y$ (uniformly in $x$)  and such that for any $x\in X$, $b_{Tx}^{-1}\circ f_{x,E_{0}}\circ b_{x}=(y\mapsto y+\beta)$, then the map $E\mapsto \rho(E):=\rho(T,f_{E})$ satisfies 
$$\forall E\in I,\ |\rho(E)-\rho(E_{0})|\leq C|E-E_{0}|$$
where $C=Lip_{E}f\cdot Lip_{y}b^{-1}$.\end{lem}
\begin{proof}Let $\ti f:X\times \R\times I\to X\times\R$ and $\ti b:X\times\R\toitself$ be lifts for $f$ and $b$ and  denote  
$\bar f:X\times\R\times I:\to X\times \R$, $(x,y,E)\mapsto (Tx, b_{Tx}^{-1}\circ \ti f_{x,E}\circ b_{x}(y),E)$. 

One can write 
$$\ti f_{x,E}(y)=y+\beta+g_{E}(x,y)$$  
and a simple calculation shows that $|g_{E}(x,y)|\leq C|E-E_{0}|$ where  $C=Lip_{E}f\cdot Lip_{y}b^{-1}$.  Comparing the orbits of $(x,0)$ under $(x,y)\mapsto (Tx,y+\beta+g_{E}(x,y))$ and $(x,y)\mapsto (Tx,y+\beta)$, and using the definition of the rotation number gives the conclusion of the Lemma.

\end{proof}
As a consequence of the preceding lemma  one gets the following lemma that applies to the case of linear cocycles (consider the associated projective dynamics):
\begin{lem}[Variation of the rotation number]Let be given $I$ an interval and  a smooth map $\cA:I\times\T^d\to SL(2,\R)$, $(E,x)\mapsto A_{E}(x)$  such that for some $E_{0}\in I$, there exists a continuous  $B:(\R/2\Z)^d\to SL(2,\R)$ for which $(0,B)^{-1}\circ(\a,A_{E_{0}})\circ(0,B)$ is a constant elliptic cocyle.  Then, for any $E\in I$ one has 
$$|\rho(\a,A_{E})-\rho(\a,A_{E_{0}})|\leq C|E-E_{0}|$$
where the constant $C$ depends only on the $C^0$-norm of $B$ and on ( the $\sup$ w.r.t. to $x$ of ) the $C^1$-norm of  of $E\mapsto A_{E}(x)$.
\end{lem}

\section{Proof of Theorem \ref{critac}}
 Since $(\a,S_{E_{*}-V})$ is almost reducible one has  $E_{*}\in \Sigma_{V,\a}$ and thus there exists $E_{0}$ close to $E_{*}$ such that $\rho(\a,S_{E_{0}-V})\in DS_{\a}\setminus\{0\}$ (the function $E\mapsto \rho(\a,{S_{E-V}})$ cannot be constant on an neighborhood of $E_{*}$ since $E_{*}\in \Sigma_{V,\alpha}$, {\it cf.} (\ref{nurho})); in particular by Theorem \ref{theo:eliasson} $(\a,S_{E_{0}-V})$ is $C^\infty$-conjugated to a constant rotation cocycle. Let $I$ be a nonempty open  interval containing $E_{0}$, $J= \{\rho(\a,S_{E-V}), E\in I\}$ and define  $f:I\to \R$ by  $f(E)=\rho(\a,S_{E-V})$. For some $\kappa>0$ the set   $M_{\kappa}:=DS_{\a}(\kappa)\cap J$ has positive Lebesgue measure. Lemma \ref{lem:var2} tells us that for any $E\in f^{-1}(M_{\kappa})$ there exists a constant $C_{E}$ such that for any $E'\in I$,  $|f(E')-f(E)|\leq C_{E}|E'-E|$. By Lemma \ref{lem:3.9} we deduce that  the set  $Z_{\kappa}$ of $E\in I$ for which $(\a,S_{E-V})$ is conjugated to a constant elliptic matrix with rotation number in $M_{\kappa}$  is of positive Lebesgue measure. In particular for $E\in Z_{\kappa}$, $L(\a,S_{E-V})=0$ and by a theorem of \cite{DS}, for Lebesgue a.e. $E\in Z_{\kappa}$
\be4\pi\sin(2\pi \rho(\a,S_{E-V}))\frac{d\rho(\a,S_{E-V})}{dE}\geq 1.\label{eq:derrho}
\ee
Let $\bar E\in Z_{\kappa} $ be a point for which the preceding inequality is satisfied.  
Now, let's write $S_{\bar E-V}(\cdot)=B(\cdot+\a)A_{0}B(\cdot)^{-1}$ for some $B\in C^\infty(\R^d/2\Z^d,SL(2,\R))$ and $A\in SO(2,\R)$. For $E-\bar E$ small enough $B(\cdot+\a)^{-1}S_{E-V}(\cdot)B(\cdot)=A_{0}+(E-\bar E)W(\cdot)$ is of the form $A_{0}e^{(E-\bar E)F_{E}(\cdot)}$. Since $\rho(A_{0})\in DS_{\a}(\kappa)$ one can find $Y_{E}\in C^\infty(\T^d,sl(2,\R))$ such that  ${\rm Ad}(A_{0})^{-1}Y(\cdot+\a)-Y(\cdot)=F(\cdot)-\hat F(0)$ and thus
$$e^{-Y_{E}(\cdot+\a)}A_{0}e^{(E-\bar E)F_{E}(\cdot)}e^{Y(\cdot)}=A_{0}e^{(E-\bar E)\hat F_{\bar E}(0)+\ti F_{E}(\cdot))}$$
where $\ti F_{E}(\cdot)$ is in $C^\infty(\T,sl(2,\R))$,  depends in an analytic way on  $E$ and satisfies $\|\ti F_{E}(\cdot)\|_{0}=O((E-\bar E)^2)$.
Since  $\rho(\a,A_{0}e^{(E-\bar E)\hat F_{\bar E}(0)})=\rho(\a,S_{E-V})-\<\a,k_{0}/2\>$ where $k_{0}\in\Z^d$ is such that  $B$ is homotopic to $R_{\<k_{0}/2,2\pi \cdot\>}$ and since $\bar E$ satisfies (\ref{eq:derrho}) one has 
$$\frac{d\rho(\a,A_{0}e^{(E-\bar E)\hat F_{\bar E}(0)})}{dE}_{|(E=\bar E)}\ne 0.$$
We now conclude by using Theorem \ref{Eliassongeneral}.

\hfill $\Box$

\bibliographystyle{plain}

\end{document}